\documentclass[a4paper,10pt
]{amsart}
\usepackage{amsmath, amscd}
\usepackage{amssymb}
\usepackage[alphabetic]{amsrefs}
\usepackage[T1]{fontenc}
\usepackage[all]{xy}
\usepackage{mathtools}
\usepackage{footmisc}
\usepackage{tikz-cd}
\usepackage{relsize} 
\usepackage[bbgreekl]{mathbbol} 
\usepackage{amsfonts}
\usepackage{comment}
\usepackage[colorlinks=true, hyperindex, linkcolor=magenta, citecolor=, pagebackref=false, urlcolor=blue, linktocpage]{hyperref}
\usepackage[color = blue!20, bordercolor = black, textsize = tiny]{todonotes}

\DeclareSymbolFontAlphabet{\mathbb}{AMSb} 
\DeclareSymbolFontAlphabet{\mathbbl}{bbold} 
\newcommand{\Prism}{{\mathlarger{\mathbbl{\Delta}}}}

\numberwithin{equation}{section}
\newtheorem{theorem}[subsection]{Theorem}
\newtheorem{corollary}[subsection]{Corollary}
\newtheorem{lemma}[subsection]{Lemma}
\newtheorem{proposition}[subsection]{Proposition}
\theoremstyle{definition}
\newtheorem{definition}[subsection]{Definition}
\newtheorem{remark}[subsection]{Remark}
\newtheorem{example}[subsection]{Example}

\newtheorem{construction}[subsection]{Construction}

\newcommand{\cC}{\mathcal{C}}

\newcommand{\cal}{\mathcal}

\DeclareMathOperator{\TC}{TC}

\DeclareMathOperator{\colim}{colim}

\newcommand{\arxivlink}[1]{\href{http://arxiv.org/abs/#1}{\texttt{arXiv:#1}}}

\newcommand{\perf}{\mathrm{perf}}

\newcommand{\lQSyn}{\mathrm{lQSyn}}
\newcommand{\QSyn}{\mathrm{QSyn}}

\title[Log TC via infinite root stacks]{Logarithmic TC via the Infinite Root Stack \\  and the Beilinson Fiber Square}

\author{Federico Binda}
\address{Dipartimento di Matematica ``Federigo Enriques'',  Universit\`a degli Studi di Milano\\ Via Cesare Saldini 50, 20133 Milano, Italy}
\email[F. Binda]{federico.binda@unimi.it}

\author{Tommy Lundemo}
\address{Mathematisch Instituut, Universiteit Utrecht, Budapestlaan 6, 3584 CD Utrecht}
\email[T. Lundemo]{t.lundemo@uu.nl}

\author{Alberto Merici}
\address{Institut f\"ur Mathematik, Universit\"at Heidelberg\\ MATHEMATIKON, Im Neuenheimer Feld 205, 69120  Heidelberg, Germany.}
\email[A. Merici]{merici@mathi.uni-heidelberg.de}

\author{Doosung Park}
\address{Department of Mathematics and Informatics, University of Wuppertal, Germany}
\email[D. Park]{dpark@uni-wuppertal.de}

\begin{document}

\begin{abstract} We apply our previous results on ``saturated descent'' to express a wide range of logarithmic cohomology theories in terms of the infinite root stack. Examples include the log cotangent complex, Rognes' log topological cyclic homology, and Nygaard-complete log prismatic cohomology. As applications, we show that the Nygaard-completion of the site-theoretic log prismatic cohomology coincides with the definition arising from log ${\rm TC}$, and we establish a log version of the ${\rm TC}$-variant of the Beilinson fiber square of Antieau--Mathew--Morrow--Nikolaus. 
\end{abstract}

\maketitle
\section{Introduction}

Recent years have seen several definitions of variants of Hochschild homology in the context of logarithmic geometry. Some of these include:

\begin{enumerate}
\item a definition proposed by Olsson \cite{Ols24} in terms of his stack of log structures \cite{Ols}; 
\item a definition in terms of perfect complexes on the infinite root stack of Talpo--Vistoli \cite{TV18}, due to Scherotzke--Sibilla--Talpo \cite{SST20};
\item a definition in terms of descent for the infinite root stack, proposed by Bhatt-Clausen-Mathew (and recorded in \cite{Mat21}, see e.g.\ \cite[Section 13]{DY24} for a realization of this idea in the case of prismatic cohomology); 
and
\item a definition due to Rognes \cite{Rog09}, which can be described as the derived self-intersections of the \emph{log diagonal} of Kato--Saito \cite[Section 4]{KS04}. 
\end{enumerate}

The last of these approaches has a fruitful generalization to the context of structured ring spectra (\cites{RSS15, RSS18}), which at present is not the case for any of the other definitions.  The present paper is concerned with the equivalence of the third and fourth definitions, which we now elaborate upon. As we show in Remark \ref{rem:sstcomp}, these differ from the second definition.

Let us first discuss the perspective of (3). Following the work of Talpo--Vistoli \cite{TV18}, one can associate to a $p$-complete pre-log ring $(A, M)$ its \emph{infinite root stack} $\sqrt[\infty]{{\rm Spf}(A, M)}$,  as we explain in Section \ref{sec:inftyroot}. This is a stack for the fpqc-topology, and as such it is a valid input for topological cyclic homology: this  is  proposed as the definition of logarithmic ${\rm TC}$ in \cite{Mat21}. 

It is \emph{a priori} unclear that this should relate to Rognes' definition ${\rm TC}((A, M) ; {\Bbb Z}_p)$. However, our previous works \cites{BLPO23Prism, BLMP24} show that ${\rm TC}((A, M) ; {\Bbb Z}_p)$ enjoys two fundamental features:

\begin{itemize}
\item[(a)]  In analogy with the filtrations on ordinary, non-logarithmic ${\rm TC}$ of Bhatt--Morrow--Scholze \cite{BMS19}, logarithmic ${\rm TC}$ admits filtrations with graded pieces described in terms of (Nygaard-completed) log prismatic cohomology $R\Gamma_{\widehat{\Prism}}({\rm Spf}(A, M))$; and
\item[(b)]  log prismatic cohomology $R\Gamma_{\widehat{\Prism}}({\rm Spf}(A, M))$ satisfies \emph{saturated descent}. We explain in Section \ref{subsec:satcechroot} that this computes log prismatic cohomology in terms of the infinite root stack. Filtrations related to those of (a) can then be used to show the same for logarithmic ${\rm TC}$ (Proposition \ref{prop:logtcsatdesc}). 
\end{itemize}

\noindent These observations accumulate to the following result:

\begin{theorem}[Section \ref{sec:tcinftyroot}]\label{thm:tcinftyroot} 
If $(A, M)$ is a quasisyntomic pre-log ring such that \begin{enumerate}
\item[(a)] $M$ is a fine, saturated, and $p$-torsion-free monoid; and
\item[(b)] the (derived) tensor product $A \widehat{\otimes}_{{\Bbb Z}_p\langle M \rangle} {\Bbb Z}_p\langle M_{\rm perf} \rangle$ is quasisyntomic.
\end{enumerate} Then there is an equivalence \[{\rm TC}((A, M) ; {\Bbb Z}_p) \simeq  {\rm TC}(\sqrt[\infty]{{\rm Spf}(A, M)} ; {\Bbb Z}_p)\]  of ${\Bbb E}_{\infty}$-rings relating Rognes' log ${\rm TC}$ and ${\rm TC}$ of the infinite root stack.
\end{theorem}

Here, ${\Bbb Z}_p\langle M \rangle$ denotes the $p$-completion of the monoid ring ${\Bbb Z}_p[M]$. In Theorem \ref{thm:tcinftyroot}, the monoid $M_{\rm perf}$ is the \emph{$p$-perfection} of $M$. For example, the $p$-perfection of the additive monoid of natural numbers ${\Bbb N}$ is ${\Bbb N}_{\rm perf} \cong {\Bbb N}[\frac{1}{p}]$.

Crucially, the assumptions of Theorem \ref{thm:tcinftyroot} hold for all $p$-complete polynomial pre-log rings $({\Bbb Z}_p\langle x_1, \dots, x_n, y_1, \dots, y_m \rangle, \langle y_1, \dots, y_m \rangle)$. In light of this, we may consider the left Kan extension \begin{equation}\label{eq:leftkanroot} \sqrt[\infty]{{\rm TC}}((-, -) ; {\Bbb Z}_p) \colon {\rm Ani}({\rm lPoly}^\wedge_{{\Bbb Z}_p}) \to {\rm Sp}_p\end{equation} of the functor ${\rm TC}(\sqrt[\infty]{{\rm Spf}(-, -)} ; {\Bbb Z}_p)$ from $p$-complete polynomial pre-log rings to all $p$-complete animated pre-log rings. Here, ${\rm Sp}_p \subset {\rm Sp}$ denotes the full subcategory consisting of $p$-complete objects in spectra. Once restricted to quasisyntomic pre-log rings, Rognes' log ${\rm TC}$ is also left Kan extended from ${\rm lPoly}_{{\Bbb Z}_p}$ by Theorem \ref{thm:logtcleftkan}. The resulting extension of Rognes' log ${\rm TC}$ to all $p$-complete animated pre-log rings thus agrees with the original one for quasisyntomic pre-log rings, and Theorem \ref{thm:tcinftyroot} continues to hold for all $p$-complete animated pre-log rings, since it does for all $p$-complete polynomial pre-log rings.

\subsection{Comparison with site-theoretic log prismatic cohomology} The definition of Nygaard-complete log prismatic cohomology $\widehat{\Prism}$ of \cite{BLPO23Prism} is obtained by mimicking the construction of \cite{BMS19} in the log setting. There is a definition of (derived) log prismatic cohomology $\Prism$ that more closely mirrors the site-theoretic approach of Bhatt--Scholze \cite{BS22}, due to Koshikawa \cite{Kos22} and further developed by Koshikawa--Yao \cite{KY23} and Diao--Yao \cite{DY24}. As we discuss in Section \ref{sec:sitecomp}, the point of comparing the two approaches by means of Nygaard-completion (in analogy with \cite[Theorem 13.1]{BS22}) has remained subtle due to the elusive nature of the Nygaard filtration in the log setting; in particular, the Nygaard filtration in the log setting may fail to be discrete even in the log quasiregular semiperfectoid case (cf.\ \cite[Example 5.2]{KY23}). Theorem \ref{thm:tcinftyroot} sheds light on this, and finally puts us in a position to prove:

\begin{theorem}[Section \ref{sec:sitecomp}]\label{thm:nygaard} 
Let $(A, I)$ be a perfect prism. There is a map \[R\Gamma_{\Prism}(- / A) \to R\Gamma_{\widehat{\Prism}}(- / A) \simeq R\Gamma_{\widehat{\Prism}}(-) \] which exhibits the target as the Nygaard-completion of the source. 
\end{theorem}

As the proof will show, the map of Theorem \ref{thm:nygaard} is functorial as the base prism $(A, I)$ is fixed. An analog of the stronger \cite[Theorem 13.1]{BS22}, where the functoriality does not depend on such a choice, would require an analog of \cite[Lemma 13.2]{BS22} in the log setting. 

\subsection{A log variant of the Beilinson fiber square} 
Let us recall that the Beilinson fiber square of Antieau--Mathew--Morrow--Nikolaus takes the form \begin{equation}\label{eq:originalbeilinson}\begin{tikzcd}K(R ; {\Bbb Q}_p) \ar{r} \ar{d} & K(R/p ; {\Bbb Q}_p) \ar{d} \\ {\rm HC}^-(R ; {\Bbb Q}_p) \ar{r} & {\rm HP}(R ; {\Bbb Q}_p)\end{tikzcd}\end{equation} for commutative rings $R$ that are henselian along $(p)$, cf.\ \cite[Theorem A]{AMMN22}. This is used to study $p$-adic deformations of $K$-theory classes (see \cite{BEK14}), and comes to life from an analogous diagram with $K$-theory replaced by ${\rm TC}$ and the Dundas--Goodwillie--McCarthy theorem \cite{DGM13}. We have the following variant of \eqref{eq:originalbeilinson}:

\begin{theorem}[Section \ref{sec:logbeilinson}]\label{thm:logbeilinson} Let $(R, P)$ be a pre-log ring. There is a square of the form \[\begin{tikzcd}{\rm TC}((R, P) ; {\Bbb Z}_p) \ar{r} \ar{d} & {\rm TC}((R \otimes_{\Bbb S} {\Bbb F}_p, P) ; {\Bbb Z}_p) \ar{d} \\ {\rm HC}^-((R, P) ; {\Bbb Z}_p) \ar{r} & {\rm HP}((R, P) ; {\Bbb Z}_p).\end{tikzcd}\] The square becomes cartesian after inverting $p$, and the canonical map  \[{\rm TC}((R \otimes_{\Bbb S} {\Bbb F}_p, P) ; {\Bbb Q}_p) \xrightarrow{} {\rm TC}((R/p, P) ; {\Bbb Q}_p)\] is an equivalence. 
\end{theorem}

The techniques leading to Theorem \ref{thm:logbeilinson} do not directly apply to give a $K$-theory version of the Beilinson fiber square in the log setting. In fact, the results of this paper suggest that the core obstruction to defining algebraic $K$-theory in the general logarithmic context is its lack of quasisyntomic descent. We are nonetheless investigating alternative ways of having algebraic $K$-theory interact with the infinite root stack in a manner that allows us to study $p$-adic deformations of $K$-theory classes in the semistable setting, and we intend to pursue this in future work. Let us note, however, that Theorem \ref{thm:logbeilinson} already improves upon results in the literature:

\begin{theorem}\label{thm:diaoyao}Let $(R, P)$ be a pre-log ring. The square \[\begin{tikzcd}{\Bbb Q}_p(n)(R, P) \ar{r} \ar{d} & {\Bbb Q}_p(n)(R/p, P) \ar{d} \\ L\Omega^{\ge n}_{(R, P) / {\Bbb Z}_p}\{n\} \otimes_{{\Bbb Z}_p} {\Bbb Q}_p \ar{r} & L\Omega_{(R, P) / {\Bbb Z}_p}\{n\} \otimes_{{\Bbb Z}_p} {\Bbb Q}_p\end{tikzcd}\] is cartesian.
\end{theorem}

If derived log de Rham cohomology is replaced with its Hodge-completion, this is \cite[Theorem 13.8]{DY24}, so that Theorem \ref{thm:diaoyao} is a  generalization of \emph{loc.\ cit.} 

\subsection{Acknowledgments} We thank Teruhisa Koshikawa for mentioning to us the idea of relating the saturated fiber product with the definition of the infinite root stack. F.B.
is partially supported by the PRIN 2022 ``The arithmetic of motives and L-functions''
at MUR (Italy). A.M. was partially
supported by The European Commission – Horizon-MSCA-PF-2022 ``Motivic integral
p-adic cohomologies'' and the Collaborative Research Centre TRR 326 GAUS - Geometry and Arithmetic of Uniformized Structures. T.L and D.P.  are partially supported by the research training group GRK 2240 ``Algebro-Geometric Methods in Algebra, Arithmetic and Topology''. T.L. is partially supported by the NWO-grant VI.Veni.242.129.  Finally, we thank the anonymous referee for a helpful report. 

\subsection{Outline} In Section \ref{sec:recollect} we recall necessary prerequisite material on log ${\rm TC}$ and prismatic cohomology. In Section \ref{sec:satdescent} we discuss saturated descent and explain its relationship to the infinite root stack, while in Section \ref{sec:tcinftyroot} we prove Theorem \ref{thm:tcinftyroot}. This is used in Section \ref{sec:sitecomp} to prove Theorem \ref{thm:nygaard} and in Section \ref{sec:logsynleftkan} to exhibit log syntomic cohomology as a left Kan extension. This, in turn, is used in the final Section \ref{sec:logbeilinson} to establish Theorems \ref{thm:logbeilinson} and \ref{thm:diaoyao}. 

\section{Recollections on log \texorpdfstring{${\rm TC}$}{TC} and prismatic cohomology}\label{sec:recollect}  We provide a brief recollection of Rognes' approach to log topological Hochschild homology \cite{Rog09} and the resulting definition of Nygaard-complete log prismatic cohomology \cite{BLPO23Prism}. We refer to \cite[Section 2]{BLPO23Prism} for a more detailed exposition. 

\subsection{The replete bar construction} Fix a ground animated commutative monoid $P$. The \emph{cyclic bar construction} \[B^{\rm cyc}_P(-) \colon {\rm Ani}({\rm CMon})_{P/} \to {\rm Ani}({\rm CMon})_{P/}\] is the tensoring $S^1 \otimes_P -$.  For any $M \in {\rm Ani}({\rm CMon})_{P/}$, the maps $* \to S^1 \to *$ exhibits $B^{\rm cyc}_P(M)$ as an augmented $M$-algebra; that is, as an object of ${\rm Ani}({\rm CMon})_{M // M}$. This additional structure is used to define the \emph{replete bar construction} $B^{\rm rep}_P(M)$ as the pullback \begin{equation}\label{eq:reppull}\begin{tikzcd}B^{\rm rep}_P(M) \ar{r} \ar{d} & B^{\rm cyc}_P(M)^{\rm gp} \ar{d} \\ M \ar{r} & M^{\rm gp}\end{tikzcd}\end{equation} of animated commutative monoids, where the bottom horizontal map is the canonical map from $M$ to its group completion (realized, for instance, as the unit of the suspension/loop-adjunction of the pointed category ${\rm Ani}({\rm CMon})$). 

\subsection{Animated log rings} We now define ${\rm Ani}({\rm PreLog})$ as the Grothendieck construction/unstraightening of the functor \begin{equation}\label{eq:pseudofunctor}{\rm Ani}({\rm CMon}) \to {\rm Cat}_{\infty}, \qquad P \mapsto {\rm Ani}({\rm CRing})_{{\Bbb Z}[P] / }.\end{equation} More concretely, then, an animated pre-log ring is a pair $(R, P)$ with $R$ an animated commutative ring, $P$ an animated commutative monoid, together with ${\Bbb Z}[P] \to R$ a map of animated commutative rings. This determines a map $\beta \colon P \to (R, \cdot)$ to the underlying multiplicative animated commutative monoid of $R$. We say that $(R, P)$ is an \emph{animated log ring} if $\beta^{-1}{\rm GL}_1(R) \simeq {\rm GL}_1(R)$.\footnote{We warn the reader who chooses to read ``animated'' as ``simplicial'' that the category of animated log rings is \emph{not} modelled by that of simplicial objects in log rings, cf.\ \cite[Remark 3.8]{SSV16}. Rather, the category of animated log rings is a localization of that of animated pre-log rings, cf.\ \cite[Section 3.2]{SSV16}.} The inclusion of animated log rings into animated pre-log rings admits a left adjoint under which all constructions we consider in this paper are invariant. In \cite[Section 2.2]{BLPO23Prism} the category ${\rm Ani}({\rm PreLog})$ is constructed as the animation of the category of polynomial pre-log rings $({\Bbb Z}[x_1, \dots, x_n, y_1, \dots y_m], \langle y_1, \dots, y_m \rangle),$ and in \cite[Remark 5.2]{Lun23} a concrete construction of the functor \eqref{eq:pseudofunctor} is provided.

\subsection{Log topological Hochschild homology} Let $(R, P)$ be an animated pre-log ring. Following Rognes \cite[Section 8]{Rog09}, we define \emph{log topological Hochschild homology} as the relative tensor product \[{\rm THH}(R, P) := {\rm THH}(R) \otimes_{{\Bbb S}[B^{\rm cyc}(P)]} {\Bbb S}[B^{\rm rep}(P)]\] of ${\Bbb E}_{\infty}$-rings, where the morphism ${\Bbb S}[B^{\rm cyc}(P)] \to {\Bbb S}[B^{\rm rep}(P)]$ is induced by the universal property of the pullback \eqref{eq:reppull} along the canonical map $B^{\rm cyc}(P) \to B^{\rm cyc}(P)^{\rm gp}$ and the collapse map $B^{\rm cyc}(P) \to P$. This definition has an obvious relative generalization ${\rm THH}((A, M) / (R, P))$ by forming all cyclic bar constructions in the relevant comma-categories, and this construction enjoys the usual transitivity property ${\rm THH}((B, N) / (A, M)) \simeq A \otimes_{{\rm THH}((A, M) / (R, P))} {\rm THH}((B, N) / (R, P))$ for a composite $(R, P) \to (A, M) \to (B, N)$ \cite[Lemma 5.4]{Lun21}. 

\begin{lemma}\label{lem:commuteswithcolims} The functor \[{\rm THH}(-, -) \colon {\rm Ani}({\rm PreLog}) \to {\rm Sp}\] commutes with sifted colimits. 
\end{lemma}

\begin{proof} The functors $A \mapsto {\rm THH}(A)$ and $M \mapsto {\Bbb S}[B^{\rm cyc}(M)] \cong {\rm THH}({\Bbb S}[M])$ commute (as functors from commutative rings/monoids to ${\Bbb E}_{\infty}$-rings) with sifted colimits. Since ${\rm THH}(A, M) = {\rm THH}(A) \otimes_{{\Bbb S}[B^{\rm cyc}(M)]} {\Bbb S}[B^{\rm rep}(M)]$ and $B^{\rm rep}(M) \cong M \oplus BM^{\rm gp}$ (cf.\ \cite[Lemma 3.17]{Rog09}), we find that log ${\rm THH}$ commutes with sifted colimits as a functor to ${\Bbb E}_{\infty}$-rings. Since the functor ${\rm CAlg}({\rm Sp}) \to {\rm Sp}$ commutes with sifted colimits, we conclude. 
\end{proof}

\begin{remark}In analogy with the description ${\rm THH}(A) = A \otimes_{A \otimes A} A$ of topological Hochschild homology as the derived self-intersections of the diagonal, there is an identification ${\rm THH}(A, M) \simeq A \otimes_{(A \otimes A)^{\rm rep}} A$, where $(A \otimes A)^{\rm rep} \to A$ is (a spectral analog of) the \emph{log diagonal} of Kato--Saito \cite[Section 4]{KS04}. For details, we refer to \cite[Section 13]{Rog09}, \cite[Proposition 5.8]{Lun21}, and \cite[Proposition 1.4]{BLPO23}. 
\end{remark}

\begin{remark} 
While phrased as an invariant of pre-log rings, ${\rm THH}(-, -)$ is invariant under the passage to the associated (derived) log structure \cite[Theorem 4.24]{RSS15}. In particular, this can be used to show that all invariants of pre-log rings considered here are invariant under passing to the associated log structure for integral monoids. Alternatively, this can also be deduced from logification invariance of the cotangent complex (see e.g.\ \cite[Lemma 3.11]{BLPO23}). 
\end{remark}

\subsection{Log topological cyclic homology} As explained in \cite[Construction 3.9]{BLPO23Prism}, the spectrum ${\rm THH}(R, P)$ admits a cyclotomic structure, and we thus obtain natural definitions of ${\rm TC}^-, {\rm TP}$ and ${\rm TC}$ of pre-log rings. Similarly, the linearization ${\rm HH}(A, M) \simeq {\Bbb Z} \otimes_{{\rm THH}({\Bbb Z})} {\rm THH}(A, M)$ recovers the definition of log Hochschild homology of \cite[Definition 3.23]{Rog09} (cf.\ \cite[Corollary 3.5]{BLPO23Prism}), and gives rise to definitions of log ${\rm HC}^-$ and ${\rm HP}$. These theories admit filtrations with graded pieces described in terms of derived log de Rham cohomology \cite[Theorem 1.3]{BLPO23Prism}, in analogy with the main results of \cite{Ant19}. 

\subsection{The log cotangent complex} Let $(R, P)$ be a (discrete) pre-log ring. The \emph{log differentials} is the $R$-module \[\Omega^1_{(R, P)} := \frac{\Omega^1_R \oplus (R \otimes_{{\Bbb Z}} P^{\rm gp})}{d\beta(p) \sim \beta(p) \otimes p},\] from which the \emph{(Gabber) log cotangent complex} ${\Bbb L}_{(R, P)}$ is obtained by left Kan extension from polynomial pre-log rings $(R, P) = ({\Bbb Z}[x_1, \dots, x_n, y_1, \dots, y_m], \langle y_1, \dots, y_m \rangle)$. 

For a general map $(R, P) \to (A, M)$ of animated pre-log rings, we define the relative log cotangent complex ${\Bbb L}_{(A, M) / (R, P)}$ to be the cofiber of the canonical map $A \otimes_R {\Bbb L}_{(R, P)} \to {\Bbb L}_{(A, M)}$ of $A$-modules. We refer to \cite[Section 2.7]{BLPO23Prism} and the references therein for further details, among which we need the cofiber sequence \begin{equation}\label{eq:amazingsequence}A \otimes_{{\Bbb Z}} (M^{\rm gp} / P^{\rm gp}) \to {\Bbb L}_{(A, M) / (R, P)} \to {\Bbb L}_{A / R \otimes_{{\Bbb Z}[P]} {\Bbb Z}[M]}\end{equation} of $A$-modules. 

\subsection{The log quasisyntomic site} Suppose now that $(R, P)$ is a discrete pre-log ring with $P$ \emph{integral}; that is, the canonical map $P \to P^{\rm gp}$ is an injection. We say that $(R, P)$ is \emph{log quasisyntomic} if $R$ has bounded $p^\infty$-torsion and the log cotangent complex ${\Bbb L}_{(R, P) / {\Bbb Z}_p}$ has $p$-complete (homological) ${\rm Tor}$-amplitude in $[0, 1]$. A log quasisyntomic pre-log ring $(S, Q)$ is \emph{log quasiregular semiperfectoid} if $S$ admits a map from a perfectoid ring and the commutative monoids $(S, \cdot)$ and $Q$ are \emph{semiperfect}; that is, their $p$-power maps are surjective. 

As explained in \cite[Section 4]{BLPO23Prism} (see also \cite[Section 3]{KY23}), the category ${\rm lQRSPerfd}$ admits the structure of a site by mimicking the approach of \cite{BMS19} with the above definitions, and there is an analog of the unfolding equivalence \cite[Theorem 4.30]{BLPO23Prism}. For any pre-log ring $(R, P)$, we write ${\rm lQSyn}_{(R, P)}$ for the category of log quasisyntomic $(R, P)$-algebras; that is, $(R, P)$-algebras $(A, M)$ whose underlying pre-log ring is log quasisyntomic. 

\subsection{Nygaard-complete log prismatic cohomology} If $(A, M)$ is log quasisyntomic, we set \[\widehat{\Prism}_{(A, M)} := R\Gamma_{\rm lqsyn}((A, M), \pi_0{\rm TC}^-((-, -) ; {\Bbb Z}_p)).\] By construction, this ${\Bbb E}_{\infty}$-algebra in $D({\Bbb Z}_p)$ is complete with respect to the \emph{Nygaard filtration}; the abutment filtration of the homotopy fixed point spectral sequence computing ${\rm TC}^- := {\rm THH}^{hS^1}$. The results of \cite{BLPO23Prism} and \cite{BLMP24} show that $\widehat{\Prism}_{(A, M)}$ participates in the expected comparison isomorphisms, such as analogs of the de Rham and crystalline comparison.

\section{Saturated descent via the infinite root stack} \label{sec:satdescent} We now recall some material on saturated descent from \cite{BLMP24}, and explain its relation to the infinite root stack. In particular, we prove that the log cotangent complex is the cotangent complex of the infinite root stack in most cases of interest (Lemma \ref{lem:diffroot}). For expositional reasons, we reformulate the relevant proofs from \cite{BLMP24} to our context; consequently, the present exposition is independent of \emph{loc.\ cit.} 

\subsection{The saturated \v{C}ech nerve} Let $\varphi \colon P \to M$ be a morphism of saturated monoids. We say that $\varphi$ is \emph{Kummer} if it is injective and for every $m \in M$, it comes from $\varphi$ up to some positive integer multiple: that is, there is $p \in P$ such that $\varphi(p) = n \cdot m$ for some non-negative integer $n$. 

The inclusion of saturated monoids into all commutative monoids admits a left adjoint. Given a Kummer morphism $\varphi \colon P \to M$, we may consider its \v{C}ech nerve \begin{equation}\label{eq:satcechnerve}\begin{tikzcd}P \ar{r} & M \ar[shift left = 0.5]{r} \ar[shift right = 0.5]{r} \!\! &\! M \oplus_P^{\rm sat} M  \ar{r} \ar[shift left = 1]{r} \ar[shift right = 1]{r} \!\!&\! M \oplus_P^{\rm sat} M \oplus_P^{\rm sat} M \cdots \end{tikzcd}\end{equation} in the category of saturated monoids. The Kummer hypothesis ensures that each of the saturated coproducts $M^{\oplus_P ^{\rm sat}\bullet + 1}$ may be described in terms of the exactification/repletion of the multiplication map $M^{\oplus_P \bullet + 1} \to M$. More explicitly, this means that there are isomorphisms $M^{\oplus_P ^{\rm sat}\bullet + 1} \cong M \oplus (M^{\rm gp}/P^{\rm gp})^{\oplus \bullet}$, where e.g.\ the morphisms $M \to M \oplus (M^{\rm gp}/P^{\rm gp})$ send $m$ to $(m, 1)$ and $(m, [m])$ respectively. For details, we refer to \cite[Section 4]{BLMP24} and Nizio{\l} \cite{Niz08}. 

\begin{example}\label{ex:nperf} Consider the morphism ${\Bbb N} \to {\Bbb N}_{\rm perf} \cong {\Bbb N}[\frac{1}{p}]$. In this case, the saturated \v{C}ech nerve  \begin{equation}\label{eq:pprufer}\begin{tikzcd}{\Bbb N} \ar{r} & {\Bbb N}[\frac{1}{p}] \ar[shift left = 0.5]{r} \ar[shift right = 0.5]{r} \!\! &\! {\Bbb N}[\frac{1}{p}] \oplus {\Bbb Z}[\frac{1}{p}]/{\Bbb Z}  \ar{r} \ar[shift left = 1]{r} \ar[shift right = 1.2]{r} \!\!&\! {\Bbb N}[\frac{1}{p}] \oplus {\Bbb Z}[\frac{1}{p}]/{\Bbb Z} \oplus {\Bbb Z}[\frac{1}{p}]/{\Bbb Z} \cdots \end{tikzcd}\end{equation} involves the $p$-Pr\"ufer group ${\Bbb Z}[\frac{1}{p}]/{\Bbb Z}$. 
\end{example}

For a $p$-complete ground ring $R$ and a pre-log ring $(A, P)$, we associate to the saturated \v{C}ech nerve \eqref{eq:satcechnerve} the diagram \begin{equation}\label{eq:satcechnervering}\begin{tikzcd} A \widehat{\otimes}_{R \langle P \rangle} R\langle M \rangle \ar[shift left = 0.5]{r} \ar[shift right = 0.5]{r} \!\! &\! A \widehat{\otimes}_{R\langle P \rangle} R\langle M  \oplus_P^{\rm sat} M\rangle  \ar{r} \ar[shift left = 1]{r} \ar[shift right = 1]{r} \!\!&\!  A \widehat{\otimes}_{R\langle P \rangle} R\langle M  \oplus_P^{\rm sat} M \oplus_P^{\rm sat} M\rangle\cdots \end{tikzcd}\end{equation}

\subsection{A convenient subcategory of descendable objects} We now isolate a subcategory of ${\rm lQSyn}_R$ for which $p$-complete log invariants will satisfy descent along the diagram \eqref{eq:pprufer}. 

\begin{definition} We write ${\rm lQSyn}_R^{\rm qsyn}$ for the full subcategory of ${\rm lQSyn}_R$ spanned by log quasisyntomic $(A, M)$ that further satisfy 
\begin{enumerate}
\item $M$ is fine, saturated, and $p$-torsionfree; and
\item $A \widehat{\otimes}_{R\langle M \rangle} R \langle M_{\rm perf} \rangle \in {\rm QSyn}_R$.  
\end{enumerate}
\end{definition}

The second condition implies in particular that the (implicitly derived) tensor product $A \otimes_{R[M]} R[M_{\rm perf}]$ is $p$-completely discrete. Observe also that ${\rm lQSyn}_R^{\rm qsyn}$ contains all $p$-complete polynomial pre-log rings $(R\langle x_1, \dots, x_n, y_1, \dots, y_m \rangle, \langle y_1, \dots, y_m \rangle)$.

One crucial consequence of the hypothesis on $M$ is that $M \to M_{\rm perf}$ is Kummer, so that \cite[Lemma 3.28]{Niz08} applies to give a homotopy equivalence \begin{equation}\label{eq:cosimpeq}{\Bbb Z}[M] \xrightarrow{\simeq} {\Bbb Z}[M_{\rm perf}^{\oplus_M^{\rm sat} \bullet + 1}]\end{equation} of cosimplicial ${\Bbb Z}[M]$-modules. Moreover, since $M_{\rm perf}^{\oplus_M^{\rm sat} \bullet + 1} \cong M_{\rm perf} \oplus (M_{\rm perf}^{\rm gp} / M^{\rm gp})^{\oplus \bullet}$ is levelwise semiperfect, the canonical map \begin{equation}\label{eq:semiperfectmon}{\Bbb L}_{{\Bbb Z}[M_{\rm perf}^{\oplus_M^{\rm sat} \bullet + 1 }]} \xrightarrow{} {\Bbb L}_{({\Bbb Z}[M_{\rm perf}^{\oplus_M^{\rm sat} \bullet + 1 }], M_{\rm perf}^{\oplus_M^{\rm sat} \bullet + 1})}\end{equation} is an equivalence after $p$-completion, as one observes by reducing mod $p$. 

\subsection{Saturated descent for the cotangent complex} We now establish the following preliminary result, which can be seen as a special case of \cite[Theorem 4.12]{BLMP24}:

\begin{proposition}\label{prop:cotcxsatdesc}The functor \[{\rm lQSyn}_R^{\rm qsyn} \to {\rm Mod}_R, \quad (A, M) \mapsto {\Bbb L}_{(A, M) / R}\] satisfies \emph{saturated descent}; that is, both maps \[{\Bbb L}_{(A, M) / R} \xrightarrow{} \lim_\Delta {\Bbb L}_{(A \widehat{\otimes}_{R\langle M \rangle} R \langle M_{\rm perf}^{\oplus_M^{\rm sat} \bullet + 1} \rangle, M_{\rm perf}^{\oplus_M \bullet + 1}) / R} \xleftarrow{}  \lim_\Delta {\Bbb L}_{A \widehat{\otimes}_{R\langle M \rangle} R \langle M_{\rm perf}^{\oplus_M^{\rm sat} \bullet + 1} \rangle / R}\] are equivalences after $p$-completion. 
\end{proposition}

For the convenience of the reader, we spell out the proof in this setting. Let us begin with the following elementary observation: 

\begin{lemma}\label{lem:cofibervanishes} For $(A, M) \in {\rm lQSyn}_R^{\rm qsyn}$, the totalization of \[{\Bbb L}_{(A \otimes_{{\Bbb Z}[M]} {\Bbb Z} [M_{\rm perf}^{\oplus_M^{\rm sat} \bullet + 1} ], M_{\rm perf}^{\oplus_M^{\rm sat} \bullet + 1}) / (A, M)}\] vanishes. 
\end{lemma}

\begin{proof} By \eqref{eq:amazingsequence}, there is a cofiber sequence \[\begin{tikzcd}[row sep = small](A \otimes_{{\Bbb Z}[M]} {\Bbb Z}[M_{\rm perf}^{\oplus_M^{\rm sat} \bullet + 1}]) \otimes_{{\Bbb Z}}^{\Bbb L} (M_{\rm perf}^{\rm gp}/M^{\rm gp})^{\bullet + 1} \ar{d} \\ {\Bbb L}_{(A, M) / R} \ar{d} \\ {\Bbb L}_{A \otimes_{{\Bbb Z}[M]} {\Bbb Z}[M_{\rm perf}^{\oplus_M^{\rm sat} \bullet + 1}] / A \otimes_{{\Bbb Z}[M]} {\Bbb Z}[M_{\rm perf}^{\oplus_M^{\rm sat} \bullet + 1}]}\end{tikzcd}\] of $A$-modules. The bottom term obviously vanishes, while the upper term vanishes as $(M^{\rm gp}_{\rm perf} / M^{\rm gp})^{\oplus \bullet + 1}$ is the \v{C}ech nerve associated to $0 \to M_{\rm perf}^{\rm gp} / M^{\rm gp}$. 
\end{proof}

\begin{proof}[Proof of Proposition \ref{prop:cotcxsatdesc}] We apply the transitivity property of the log cotangent complex to the composition $R \to (A, M) \to (A \otimes_{{\Bbb Z}[M]} {\Bbb Z}[M_{\rm perf}^{\oplus_M^{\rm sat} \bullet + 1}], M_{\rm perf}^{\oplus_M^{\rm sat} \bullet + 1})$ to obtain the cofiber sequence \[\begin{tikzcd}[row sep = small]{\Bbb Z}[M_{\rm perf}^{\oplus_M^{\rm sat} \bullet + 1}] \otimes_{{\Bbb Z}[M]} {\Bbb L}_{(A, M) / R} \ar{d}  \\ {\Bbb L}_{(A \otimes_{{\Bbb Z}[M]} {\Bbb Z}[M_{\rm perf}^{\oplus_M^{\rm sat} \bullet + 1}], M_{\rm perf}^{\oplus_M^{\rm sat} \bullet + 1}) / R} \ar{d} \\ {\Bbb L}_{(A \otimes_{{\Bbb Z}[M]} {\Bbb Z} [M_{\rm perf}^{\oplus_M^{\rm sat} \bullet + 1} ], M_{\rm perf}^{\oplus_M^{\rm sat} \bullet + 1}) / (A, M)}\end{tikzcd}\] of cosimplicial $A \otimes_{{\Bbb Z}[M]} {\Bbb Z} [M_{\rm perf}^{\oplus_M^{\rm sat} \bullet + 1}]$-modules. The totalization of the cofiber term vanishes by Lemma \ref{lem:cofibervanishes}. The map ${\Bbb Z}[M] \to {\Bbb Z}[M_{\rm perf}^{\oplus^{\rm sat}_M \bullet + 1}]$ is a cosimplicial homotopy equivalence by  \eqref{eq:cosimpeq}, so that the fiber term is equivalent to ${\Bbb L}_{(A, M) / R}$. This establishes the first equivalence. The second follows from \eqref{eq:semiperfectmon}. 
\end{proof}

\subsection{Saturated descent for log topological cyclic homology} We now explain how the results of \cites{BLPO23Prism, BLMP24} apply to compute log ${\rm TC}$ in terms of ordinary ${\rm TC}$ of the diagram \eqref{eq:satcechnervering}. The next result states that log ${\rm TC}$ can be computed by descent along the diagram \eqref{eq:satcechnervering} with respect to the Kummer map $M \to M_{\rm perf}$:

\begin{proposition}\label{prop:logtcsatdesc} For $(A, M) \in {\rm lQSyn}^{{\rm qsyn}}$, the canonical maps \[\begin{tikzcd}[row sep = tiny]{\rm TC}((A, M) ; {\Bbb Z}_p) \ar{d}  
\\ \lim_\Delta {\rm TC}((A \otimes_{{\Bbb Z}[M]} {\Bbb Z}[M_{\rm perf}^{\oplus_M^{\rm sat} \bullet + 1}], M_{\rm perf}^{\oplus_M^{\rm sat} \bullet + 1}) ; {\Bbb Z}_p) \\  \lim_\Delta {\rm TC}(A \otimes_{{\Bbb Z}[M]} {\Bbb Z}[M_{\rm perf}^{\oplus_M^{\rm sat} \bullet + 1}] ; {\Bbb Z}_p) \ar{u} \end{tikzcd}\]
are equivalences. 
\end{proposition}

The proof of Proposition \ref{prop:logtcsatdesc} requires some preparation. Let us first establish the analogous result for ${\rm THH}$:

\begin{lemma}\label{lem:thhsatdesc} For $(A, M) \in {\rm lQSyn}^{\rm qsyn}$, the canonical maps \[\begin{tikzcd}[row sep = tiny]{\rm THH}((A, M) ; {\Bbb Z}_p) \ar{d}  
\\ \lim_\Delta {\rm THH}((A \otimes_{{\Bbb Z}[M]} {\Bbb Z}[M_{\rm perf}^{\oplus_M^{\rm sat} \bullet + 1}], M_{\rm perf}^{\oplus_M^{\rm sat} \bullet + 1}) ; {\Bbb Z}_p) \\  \lim_\Delta {\rm THH}(A \otimes_{{\Bbb Z}[M]} {\Bbb Z}[M_{\rm perf}^{\oplus_M^{\rm sat} \bullet + 1}] ; {\Bbb Z}_p) \ar{u} \end{tikzcd}\] are equivalences.
\end{lemma}

\begin{proof} Inspired by \cite[Proof of Corollary 3.4]{BMS19}, we consider the diagram \[\begin{tikzcd}[row sep = tiny]\lim_n\tau_{\le n}{\rm THH}({\Bbb Z}) \otimes_{{\rm THH}({\Bbb Z})} {\rm THH}((A, M) ; {\Bbb Z}_p) \ar{d}  
\\ \lim_\Delta  \lim_n \tau_{\le n}{\rm THH}({\Bbb Z}) \otimes_{{\rm THH}({\Bbb Z})} {\rm THH}((A \otimes_{{\Bbb Z}[M]} {\Bbb Z}[M_{\rm perf}^{\oplus_M^{\rm sat} \bullet + 1}], M_{\rm perf}^{\oplus_M^{\rm sat} \bullet + 1}) ; {\Bbb Z}_p) \\  \lim_\Delta \lim_n \tau_{\le n}{\rm THH}({\Bbb Z}) \otimes_{{\rm THH}({\Bbb Z})} {\rm THH}(A \otimes_{{\Bbb Z}[M]} {\Bbb Z}[M_{\rm perf}^{\oplus_M^{\rm sat} \bullet + 1}] ; {\Bbb Z}_p) \ar{u} \end{tikzcd}\] involving weak Postnikov towers \cite[Lemma 3.3]{BMS19}, which is equivalent to the diagram of the lemma. By induction, this allows us to reduce to the analogous statement for ${\rm THH}((-, -) ; {\Bbb Z}_p) \otimes_{{\rm THH}({\Bbb Z})} {\Bbb Z} \simeq {\rm HH}((-, -) ; {\Bbb Z}_p)$ \cite[Corollary 3.4]{BLPO23Prism} (it is not necessary to $p$-complete the tensor product as the homotopy groups $\pi_n{\rm THH}({\Bbb Z})$ are finite for $n > 0$). This holds, in turn, by the log HKR-filtration \cite[Theorem 1.1]{BLPO23} and Proposition \ref{prop:cotcxsatdesc}.
\end{proof} 

\begin{lemma}\label{lem:thhtcpsatdesc}  The functor \[{\rm THH}((-, -) ; {\Bbb Z}_p)^{tC_p} \colon {\rm lQSyn}^{\rm qsyn} \to {\rm Sp}_p\] satisfies saturated descent.
\end{lemma}

\begin{proof}
Variants of the arguments used to reduce Lemma \ref{lem:thhsatdesc} to saturated descent for the cotangent complex are also applicable in this case, cf.\ \cite[Remark 3.5]{BMS19}.  
\end{proof}

\begin{corollary}\label{cor:thhsatdesccyc} The functor \[{\rm THH}((-, -) ; {\Bbb Z}_p) \colon {\rm lQSyn}^{\rm qsyn} \to {\rm CycSp}_p\] satisfies saturated descent.
\end{corollary}

\begin{proof} By \cite[Definition II.1.6]{NS18}, this follows from Lemmas \ref{lem:thhsatdesc} and \ref{lem:thhtcpsatdesc}.
\end{proof}

\begin{proof}[Proof of Proposition \ref{prop:logtcsatdesc}] This follows from the description of ${\rm TC}((A, M) ; {\Bbb Z}_p)$ as \[{\rm TC}((A, M) ; {\Bbb Z}_p) := {\rm Map}_{{\rm CycSp}_p}({\Bbb S}^{\rm triv}, {\rm THH}((A, M) ; {\Bbb Z}_p))\] and Corollary \ref{cor:thhsatdesccyc}.
\end{proof}

\subsection{The infinite root stack}\label{sec:inftyroot} We give a very brief recollection of the construction of the ($p$-typical variant of the) infinite root stack of Talpo--Vistoli \cite{TV18}. The only case of interest to us is that covered by \cite[Proposition 3.10]{TV18}, which we discuss in Example \ref{ex:onevariable} and Proposition \ref{prop:fscechnerve}. 

Let $({\cal X}, {\cal M})$ be a log formal scheme (see Koshikawa \cite[Appendix A]{Kos22} for material on log formal schemes). The \emph{infinite root stack} $\sqrt[\infty]{({\cal X}, {\cal M})}$ is an fpqc-stack over ${\rm Aff}_{\cal X}$ fibered in groupoids, determined by the following property: Given a solid arrow diagram \[\begin{tikzcd}\vspace{10 mm} & \sqrt[\infty]{({\cal X}, {\cal M})} \ar{d} \\ {\rm Spf}(A) \ar{r}{f} \ar[dashed]{ur} & {\cal X},\end{tikzcd}\] specifying a lift corresponds precisely to specifying a symmetric monoidal functor $f^*({\cal M}/{\cal O}_{\cal X}^\times)_{\rm perf} \to [{\cal O}_{{\rm Spf}(A)} / {\cal O}_{{\rm Spf}(A)}^\times]$ and an isomorphism relating its pre-composition with $f^*({\cal M}/{\cal O}_{\cal X}^\times) \to f^*({\cal M}/{\cal O}_{\cal X}^\times)_{\rm perf}$ (denoted $f^*({\cal M}/{\cal O}_{\cal X}^\times)_{\infty}$ in \cite{DY24}) and the composite \[f^*({\cal M}/{\cal O}_{\cal X}^\times) \to [{\cal O}_{\cal X} / {\cal O}_{\cal X}^\times] \to [{\cal O}_{{\rm Spf}(A)} / {\cal O}_{{\rm Spf}(A)}^\times] \simeq {\rm Div}_{{\rm Spf}(A)}\] induced by the log structure and $f$. See \cite[Definition 3.3 and Lemma 3.12]{TV18}  and Diao--Yao \cite[Section 13.1]{DY24}.

\begin{remark} We need the assignment $A \mapsto {\rm Div}_{{\rm Spf}(A)}$ to satisfy quasisyntomic descent. Here, ${\rm Div}_{{\rm Spf}(A)}$ is the category of pairs $(L, s)$ with $L$ a line bundle and $s$ a global section. To see this, we consider for a quasisyntomic cover $A \to B$ the diagram \[\begin{tikzcd}{\rm Div}_{{\rm Spf}(A)} \ar{r} \ar{d} & \lim_\Delta {\rm Div}_{{\rm Spf}(B^{\otimes_{A} \bullet})} \ar{d} \\ {\rm lim}_n {\rm Div}_{{\rm Spec}(A/p^n)} \ar{r} & {\rm lim}_n \lim_\Delta {\rm Div}_{{\rm Spec}({B/p^n}^{\otimes_{A/p^n} \bullet})}.\end{tikzcd}\] Here, the vertical maps are equivalences by Grothendieck's existence theorem, while the bottom map is an equivalence by fpqc-descent. Hence the top map is an equivalence, as desired (note that quasisyntomic descent for $\mathrm{Pic}(-)$ is also recorded in \cite[Lemma 9.6]{BS22}).
\end{remark}

\subsection{The saturated \v{C}ech nerve computes the cohomology of the infinite root stack}\label{subsec:satcechroot} We now explain how the saturated \v{C}ech nerve naturally arises from the infinite root stack in cases of interest. Let us first consider the case of a polynomial ring in one variable: 

\begin{example}\label{ex:onevariable}  Let us explain how to recover the saturated descent diagram associated to ${\Bbb N} \to {\Bbb N}_{\rm perf} \cong {\Bbb N}[\frac{1}{p}]$ (cf.\ Example \ref{ex:nperf}) from the infinite root stack $\sqrt[\infty]{{\rm Spf}({\Bbb Z}_p\langle t \rangle, \langle t \rangle)}$. Applying $p$-complete monoid rings to \eqref{eq:pprufer} in this case, we obtain \[\begin{tikzcd}[column sep = tiny]{\Bbb Z}_p\langle t^{1/p^{\infty}} \rangle \ar[shift left = 0.5]{r} \ar[shift right = 0.5]{r} \!\! &\! {\Bbb Z}_p\langle t^{1/p^{\infty}} \rangle \widehat{\otimes}_{{\Bbb Z}_p} {\Bbb Z}_p \langle t^{\pm 1/p^{\infty}}/t^{\pm 1} \rangle  \ar{r} \ar[shift left = 1]{r} \ar[shift right = 1.2]{r} \!\!&\! {\Bbb Z}_p\langle t^{1/p^{\infty}} \rangle \widehat{\otimes}_{{\Bbb Z}_p} {\Bbb Z}_p \langle t^{\pm 1/p^{\infty}}/t^{\pm 1} \rangle^{\widehat{\otimes}_{{\Bbb Z}_p} 2}\cdots \end{tikzcd}\] In other words, the $n$th stage of the saturated descent diagram involves one copy of ${\Bbb Z}_p\langle {\Bbb N}[\frac{1}{p}] \rangle$ and $n$ copies of the $p$-completed group ring ${\Bbb Z}_p \langle {\Bbb Z}[\frac{1}{p}]/{\Bbb Z} \rangle$ of the $p$-Pr\"ufer group ${\Bbb Z}[\frac{1}{p}]/{\Bbb Z}$. 

We now consider the infinite root stack $\sqrt[\infty]{{\rm Spf}({\Bbb Z}_p\langle t \rangle, \langle t \rangle)}$. We may use \cite[Corollary 3.13]{TV18} to identify \[\sqrt[\infty]{{\rm Spf}({\Bbb Z}_p\langle t \rangle, \langle t \rangle)} \simeq [{\rm Spf}({\Bbb Z}_p\langle t^{1/p^{\infty}} \rangle) / \mu_{p^\infty}].\]  We may thus compute the cohomology of the infinite root stack by means of the \v{C}ech nerve of the morphism \[{\rm Spf}({\Bbb Z}_p\langle t^{1/p^{\infty}} \rangle) \to [{\rm Spf}({\Bbb Z}_p\langle t^{1/p^{\infty}} \rangle) / \mu_{p^\infty}].\] From this, we see that (${\rm Spf}$ of the $p$-complete monoid ring of) the diagram \eqref{ex:nperf} recovers the usual description of the \v{C}ech nerve of quotient stacks $X \to [X / G]$. 
\end{example}

\noindent The following generalizes Example \ref{ex:onevariable} to all cases of interest to us:

\begin{proposition}\label{prop:fscechnerve} Let $(A, M) \in {\rm lQSyn}_R^{\rm qsyn}$. The saturated \v{C}ech nerve \eqref{eq:satcechnervering} associated to the Kummer map $M \to M_{\rm perf}$ is isomorphic to the \v{C}ech nerve of the  quasisyntomic cover 
\begin{equation}\label{eq:fromspf}{\rm Spf}(A \widehat{\otimes}_{R\langle M \rangle} R\langle M_{\rm perf} \rangle) \to \sqrt[\infty]{{\rm Spf}(A, M)}\end{equation} after applying ${\rm Spf}(-)$.  
\end{proposition}

\begin{proof} This follows from (the $p$-complete variant of) \cite[Corollary 3.13]{TV18}. In this case, we have that \[\sqrt[\infty]{{\rm Spf}(A, M)} \simeq [{\rm Spf}(A \widehat{\otimes}_{R\langle M \rangle} R\langle M_{\rm perf} \rangle) / \mu_{p^{\infty}}(M)].\] Here $\mu_{p^\infty}(M)$ is the Cartier dual of $M^{\rm gp}_{\rm perf} / M^{\rm gp}$; it is isomorphic to $\mu_{p^\infty}^{\times n}$, where $n$ is the rank of the free abelian group $M^{\rm gp}$. Thus, the \v{C}ech nerve of \eqref{eq:fromspf} recovers \eqref{eq:satcechnervering} after applying ${\rm Spf}(-)$. 
\end{proof}

\begin{corollary}\label{cor:freecechnerve} Let $R$ be a $p$-complete ground ring and let \[(A, M) = (R \langle x_1, \dots, x_n, y_1, \dots, y_m \rangle, \langle y_1, \dots, y_m \rangle)\] be a $p$-complete polynomial pre-log $R$-algebra. The formal spectrum of the saturated descent diagram of $(A, M) \to (A \widehat{\otimes}_{R\langle M \rangle} R \langle M_{\rm perf} \rangle, M_{\rm perf})$  is equivalent to the \v{C}ech nerve of the morphism \[{\rm Spf}(A \widehat{\otimes}_{R\langle M \rangle} R\langle M_{\rm perf} \rangle) \to \sqrt[\infty]{{\rm Spf}(A, M)}. \] 
\end{corollary}

\begin{proof} Free commutative monoids are fine and saturated, so Proposition \ref{prop:fscechnerve} applies to conclude. 
\end{proof}

\subsection{Log differentials via the infinite root stack} The following construction explains when and how we can extend invariants of interest to the infinite root stack:

 For a $p$-complete ring $R$, recall from \cite[Definition 4.10]{Mon21} that a \emph{quasisyntomic stack} over $R$ is a stack for the quasisyntomic site $\QSyn_R^\mathrm{op}$, for example, the infinite root stack. 

\begin{definition}\label{constr:inftyroot}

Let $\cC$ be an $\infty$-category with small limits, and let $\mathbf{E}$ be a $\cC$-valued quasisyntomic sheaf on $\QSyn_R^\mathrm{op}$,
where $R$ is $p$-complete with bounded $p^\infty$-torsion.
For a quasisyntomic stack $\mathcal{X}$ over $R$,
we define
\[
\mathbf{E}(\mathcal{X})
:=
\lim_{\mathrm{Spf}(A)\to \mathcal{X}} \mathbf{E}(A),
\]
see \cite[Definition 4.13]{Mon21} for the prismatic cohomology case.

For $(A, M)\in \lQSyn_R^\mathrm{fsh}$, Proposition
\ref{prop:fscechnerve} yields an equivalence
\begin{equation}\label{eq:rootstackdef}
\mathbf{E}(\sqrt[\infty]{(A,M)})
\simeq
\lim_\Delta
\mathbf{E}(A\widehat{\otimes}_{R\langle M\rangle} R\langle M_\mathrm{perf}^{\oplus_M^\mathrm{sat}\bullet +1}\rangle).\end{equation}
Note that this construction does \emph{not} apply to algebraic $K$-theory, as it does not have quasisyntomic descent.
\end{definition} 

Since the ordinary cotangent complex $A \mapsto {\Bbb L}_{A / R}$ is an fpqc-sheaf (by e.g.\ \cite[Theorem 3.1]{BMS19}), we can apply ${\Bbb L}_{- / R}$ to the infinite root stack, as explained in Construction \ref{constr:inftyroot}. In favorable cases, this recovers the log cotangent complex:

\begin{lemma}\label{lem:diffroot} Let $(A, M) \in {\rm lQSyn}_R^{\rm qsyn}$ . Then there is a canonical equivalence \[\widehat{\Bbb L}_{(A, M) / R} \simeq \widehat{{\Bbb L}}_{\sqrt[\infty]{{\rm Spf}(A, M)} / R}\] relating the log cotangent complex and the cotangent complex of the infinite root stack.
\end{lemma}

\begin{proof} The canonical maps \[\widehat{\Bbb L}_{(A, M) / R} \xrightarrow{\simeq} \lim_\Delta \widehat{\Bbb L}_{(A \widehat{\otimes}_{R \langle M \rangle} R \langle M_{\rm perf}^{\oplus_M^{
\rm sat} \bullet + 1} \rangle, M_{\rm perf}^{\oplus_M^{\rm sat} \bullet + 1}) / R} \xleftarrow{\simeq} \lim_\Delta \widehat{\Bbb L}_{A \widehat{\otimes}_{R \langle M \rangle} R \langle M_{\rm perf}^{\oplus_M^{
\rm sat} \bullet + 1} \rangle / R}\] are equivalences by Proposition \ref{prop:cotcxsatdesc}. In other words, this computes precisely \eqref{eq:rootstackdef}. 
\end{proof}

In particular, our results apply to recover the following observation of Akhil Mathew (stated in \cite{Mat21} and recorded with proof in \cite[Lemma 13.3]{DY24}):

\begin{corollary}\label{cor:onevar} There is an equivalence \[\widehat{\Omega}^1_{(R\langle t \rangle, \langle t \rangle) / R} \simeq \widehat{{\Bbb L}}_{\sqrt[\infty]{{\rm Spf}(R\langle t \rangle, \langle t \rangle)} / R}\] relating $p$-complete log differentials and the cotangent complex of the infinite root stack. \qed
\end{corollary}

\section{Logarithmic \texorpdfstring{${\rm TC}$}{TC} via the infinite root stack}\label{sec:tcinftyroot} We now explain how saturated descent for log topological cyclic homology (Proposition \ref{prop:logtcsatdesc}) allows us to describe it in terms of the infinite root stack, accumulating in a proof of Theorem \ref{thm:tcinftyroot}. 

\subsection{Logarithmic \texorpdfstring{${\rm TC}$}{TC} as a left Kan extension} From \cite[Theorem G]{CMM21}, we obtain the following (cf.\ \cite[Proof of Theorem 5.1(ii)]{AMMN22}):

\begin{theorem}\label{thm:logtcleftkan} The functor \[{\rm TC}((-, -) ; {\Bbb Z}_p) \colon {\rm lQSyn} \to {\rm Sp}_p\] is left Kan extended from $p$-complete polynomial pre-log rings. 
\end{theorem}

\begin{proof} Using Lemma \ref{lem:commuteswithcolims}, we can argue exactly as in \cite[Proof of Theorem 2.7]{CMM21} (cf.\ \cite[Proof of Theorem 5.1(ii)]{AMMN22}). 
\end{proof}

\subsection{Topological cyclic homology of the infinite root stack} 
Topological cyclic homology ${\rm TC}(- ; {\Bbb Z}_p)$ is an fpqc-sheaf: This can be seen, for instance, by arguing as in the proof of Proposition \ref{prop:logtcsatdesc} to reduce this to the same question for ${\rm THH}$ and ${\rm THH}^{tC_p}$, which is covered by \cite[Section 3]{BMS19}. We may thus evaluate ${\rm TC}$ on the infinite root stack as in Construction \ref{constr:inftyroot}, and we obtain:

\begin{proof}[Proof of Theorem \ref{thm:tcinftyroot}] The target of the equivalence of Proposition \ref{prop:logtcsatdesc} is precisely the description of ${\rm TC}(\sqrt[\infty]{{\rm Spf}(A, M)} ; {\Bbb Z}_p)$ from \eqref{eq:rootstackdef}. 
\end{proof}

A variant more in the spirit of the notion of derived log prismatic cohomology that we shall study in the next section is the left Kan extension of  \begin{equation}\label{eq:pleasecommutewithsiftedcolimits}{\rm lPoly}_{{\Bbb Z}_p}^\wedge \to {\rm Sp_p}, \quad (A, M) \to {\rm TC}(\sqrt[\infty]{{\rm Spf}(A, M)} ; {\Bbb Z}_p)\end{equation} from $p$-complete polynomial pre-log ${\Bbb Z}_p$-algebras to $p$-complete spectra, that we denote by $\sqrt[\infty]{{\rm TC}}((-, -) ; {\Bbb Z}_p)$. We observe that, since Theorem \ref{thm:tcinftyroot} in particular applies to all $p$-complete polynomial pre-log rings, we obtain an equivalence ${\rm TC}((A, M) ; {\Bbb Z}_p) \simeq \sqrt[\infty]{{\rm TC}}((A, M) ; {\Bbb Z}_p)$ for all $p$-complete animated pre-log rings $(A, M)$. Here, we understand ${\rm TC}((A, M) ; {\Bbb Z}_p)$ to be the left Kan extension of ${\rm TC}((-, -) ; {\Bbb Z}_p)$ from $p$-complete polynomial pre-log rings, which agrees with Rognes' log ${\rm TC}$ on ${\rm lQSyn}$ by Theorem \ref{thm:logtcleftkan}.

\begin{remark}\label{rem:sstcomp}
    In \cite{SST20}, the authors considered $\TC$ of the infinite root stack as \[
    \TC^{SST}({\rm Spf}(A,M)):=\colim_m \TC(\sqrt[m]{{\rm Spf}(A,M)}).
    \]
    This does not recover our definition, even in the polynomial case; as we shall see below, the necessary interchange of (co)limits is not applicable in this situation. In fact, $\sqrt[m]{{\rm Spf}(A,M)}$ has an fpqc atlas given by ${\rm Spf}(A \widehat{\otimes}_{R \langle M \rangle} R\langle M \rangle)$, where the term on the right is given by the map $M \xrightarrow{p^m} M$. If $\check{C}^n$ denotes the $n$-term of the saturated \v Cech nerve, we have an equivalence \[
    {\rm TC}(A\widehat{\otimes}_{R\langle M\rangle}R \langle \check{C}^n(M\to M_{\perf}) \rangle) \simeq \colim_m {\rm TC}(A\widehat{\otimes}_{R\langle M\rangle}R \langle \check{C}^n(M\xrightarrow{p^m} M)\rangle),
    \] 
    for each $n$. The situation may thus be summarized by the following equivalences and the canonical map $(*)$:
    \begin{align*}
        &\TC^{SST}({\rm Spf}(A,M)):=\colim_m \TC(\sqrt[m]{{\rm Spf}(A,M)})\\
        &\simeq \colim_m \lim_n \TC(A\widehat{\otimes}_{R\langle M\rangle}R \langle \check{C}^n(M\xrightarrow{p^m} M)\rangle) \\  &\xrightarrow{(*)}
        \lim_n \colim_m \TC(A\widehat{\otimes}_{R\langle M\rangle}R \langle \check{C}^n(M\xrightarrow{p^m} M)\rangle)\\
        &\simeq \lim_n \TC(A\widehat{\otimes}_{R\langle M\rangle}R \langle \check{C}^n(M\to M_{\rm perf})\rangle) \simeq \TC(\sqrt[\infty]{{\rm Spf}(A,M)}).
    \end{align*}
    The map $(*)$ above is not an equivalence. This can be seen, for instance, using the decompositions of \cite[Corollary C]{SST20}, which are not available for Rognes' log ${\rm TC}$. Indeed, these decompositions imply that ${\rm TC}^{\rm SST}(-)$ is not $\square$-invariant, which is the case for Rognes' log ${\rm TC}$ (cf.\ \cite[Section 8]{BPO24}). 
\end{remark}

\section{Comparison with site-theoretic log prismatic cohomology}\label{sec:sitecomp} We now provide the comparison with the site-theoretic definition of log prismatic cohomology introduced by Koshikawa \cite{Kos22} and elaborated upon by Koshikawa--Yao \cite{KY23} and Diao--Yao \cite{DY24}.  In particular, we prove Theorem \ref{thm:nygaard} from the introduction. Moreover, we discuss the equivalence between the natural choice of Nygaard filtration arising from saturated descent and that pursued by Koshikawa--Yao (Section \ref{sec:nygaardcomp}). 

\subsection{Review of derived log prismatic cohomology} 
Let us briefly review the derived log prismatic cohomology introduced by Koshikawa--Yao \cite[Section 4]{KY23}. Let $(A, I, M)$ be a bounded pre-log prism in the sense of Koshikawa \cite[Definition 3.3]{Kos22}. We define a functor \[{\rm lPoly}_{(A/I, M)}^\wedge \to {\cal D}(A), \qquad (R, P) \mapsto \Prism_{(R, P) / (A, M)},\] where $\Prism_{(R, P) / (A, M)}$ is the site-theoretic definition of log prismatic cohomology of \cite[Section 4.2]{Kos22}. By left Kan extension, this gives a functor \[\Prism_{(-, -) / (A, M)} \colon {\rm Ani}({\rm PreLog}_{(A/I, M)^\wedge}) \to {\cal D}(A)\] that we shall refer to as \emph{derived log prismatic cohomology}.

\subsection{Derived log prismatic cohomology via the infinite root stack} Suppose now that $(A, I)$ is a perfect prism. We could have equally well worked with a pre-log prism $(A, I, M)$ with $M$ a perfect monoid by \cite[Proposition 4.5]{KY23} and \cite[Proposition 4.7]{BLMP24}. Extending the present discussion to more general base prisms would require additional effort. 

By quasisyntomic descent, we may apply Definition \ref{constr:inftyroot} to define $\Prism_{\sqrt[\infty]{{\rm Spf}(R, P)} / A}$ for $(R, P) \in {\rm lQSyn}_{A/I}$. 

\begin{lemma}\label{lem:prisminftyroot} With notation as above, assume further that $(R, P) \in {\rm lQSyn}_{A/I}^{\rm qsyn}$. Then there is a natural equivalence $\Prism_{(R, P) / A} \simeq \Prism_{\sqrt[\infty]{{\rm Spf}(R, P)} / A}$  relating the derived log prismatic cohomology of $(R, P)$ to that of the infinite root stack $\sqrt[\infty]{{\rm Spf}(R, P)}$. 
\end{lemma}

\begin{proof} We can compute $\Prism_{\sqrt[\infty]{{\rm Spf}(R, P)} / A}$ in terms of the ordinary prismatic cohomology of the saturated \v{C}ech nerve \eqref{eq:satcechnervering} associated to $P \to P_{\rm perf}$. This takes the form \begin{equation}\label{eq:satcechnervepperf}\begin{tikzcd}[column sep = small] R\langle P_{\rm perf} \rangle \ar[shift left = 0.5]{r} \ar[shift right = 0.5]{r} \!\! &\! R \widehat{\otimes}_{A/I
\langle P_{\rm perf} \rangle} A/I\langle P_{\rm perf}  \oplus_P^{\rm sat} P_{\rm perf}\rangle  \ar{r} \ar[shift left = 1]{r} \ar[shift right = 1]{r} \!\!&\! \cdots,\end{tikzcd}\end{equation} see \eqref{eq:rootstackdef}. Ignoring the question of Nygaard filtrations for now, our remaining task is to argue that $\Prism_{(R, P) / A}$ can be computed by the same diagram. For this, we denote by $R\langle P \rangle^{\otimes^{\rm sat} \bullet + 1}$ the diagram \eqref{eq:satcechnervepperf}, while we write $P_{\rm perf}^{\oplus^{\rm sat}_P \bullet + 1}$ for the diagram \eqref{eq:satcechnerve} associated to $\varphi \colon P \to P_{\rm perf}$. We then observe that there are canonical maps \begin{equation}\label{eq:botheqs}\Prism_{(R, P) / A} \xrightarrow{} \Prism_{(R\langle P \rangle^{\otimes^{\rm sat} \bullet + 1}, P_{\rm perf}^{\oplus^{\rm sat}_P \bullet + 1}) / A} \xleftarrow{} \Prism_{R\langle P \rangle^{\otimes^{\rm sat} \bullet + 1} / A},\end{equation} where the right-hand term computes $\Prism_{\sqrt[\infty]{{\rm Spf}(R, P)} / A}$. We claim that both maps are equivalences on totalizations. For this, we may first reduce modulo $A/I$ and argue as in \cite[Proof of Corollary 4.22]{BLMP24} and commute the base-change past the totalization. Then we use the (derived) Hodge--Tate comparison \cite[Proposition 4.5]{KY23} to reduce checking that \eqref{eq:botheqs} are equivalences by replacing (log) prismatic cohomology with the (log) cotangent complex. This is true, by saturated descent for the log cotangent complex \cite[Theorem 4.12]{BLMP24} on one hand, and the fact that the $p$-complete log cotangent complex does not see semiperfect monoids on the other \cite[Proposition 4.7]{BLMP24}. Thus the maps \eqref{eq:botheqs} are equivalences. 
\end{proof}

Let us write $\sqrt[\infty]{\Prism}_{(-, -) / A}$ for the left Kan extension of $\Prism_{\sqrt[\infty]{\rm Spf}(-, -) / A}$ from $p$-complete polynomial $A/I$-algebras to all animated $p$-complete $A/I$-algebras. We have the following generalization of \cite[Theorem 13.5]{DY24}:

\begin{theorem}\label{thm:prisminftyroot} Let $(A, I)$ be a perfect prism. The two functors \[\sqrt[\infty]{\Prism}_{(-, -) / A}, \quad {\Prism}_{(-, -) / A} \colon {\rm Ani}({\rm PreLog}_{A/I}^\wedge) \to {\cal D}(A)\] are canonically equivalent.
\end{theorem}

\begin{proof} By Lemma \ref{lem:prisminftyroot}, they agree for $p$-complete polynomial pre-log rings over $A/I$, and so the result follows from their definitions as left Kan extensions.
\end{proof}

We next want to claim that the equivalence of Theorem \ref{thm:prisminftyroot} is compatible with Nygaard filtrations. To see this, we will first introduce an alternative definition for the Nygaard filtration on the site-theoretic log-prismatic cohomology, which deviates from that pursued by Koshikawa--Yao \cite[Section 5.5]{KY23}, inspired by the equivalences \eqref{eq:botheqs}. We explain in Section \ref{sec:nygaardcomp} that the two filtrations are equivalent.

The repletion/exactification of the multiplication map $P_{\rm perf}^{\oplus \bullet + 1} \to P_{\rm perf}$ is canonically isomorphic to the projection \[P_{\rm perf}^{\oplus^{\rm ex} \bullet + 1} \cong P_{\rm perf} \oplus (P_{\rm perf}^{\rm gp})^{\oplus \bullet}  \to P_{\rm perf}.\] Similarly, the repletion/exactification of the multiplication map $P_{\rm perf}^{\oplus_P \bullet + 1} \to P_{\rm perf}$ is canonically isomorphic to the projection \[P_{\rm perf}^{\oplus_P^{\rm ex} \bullet + 1} \cong P_{\rm perf} \oplus (P_{\rm perf}^{\rm gp}/P^{\rm gp})^{\oplus \bullet}  \to P_{\rm perf}.\] This is canonically isomorphic to the saturated coproducts $P_{\rm perf}^{\oplus_P^{\rm sat} \bullet + 1} \cong P_{\rm perf} \oplus (P_{\rm perf}^{\rm gp} / P^{\rm gp})^{\oplus \bullet}$ (cf.\ \cite[Section 4]{BLMP24} or \cite{Niz08}) and receives a map from $P_{\rm perf}^{\oplus^{\rm ex} \bullet + 1}$ by means of the quotient map $P_{\rm perf}^{\rm gp} \to P_{\rm perf}^{\rm gp} / P^{\rm gp}$. If $(A, I)$ is a perfect prism and $(R, P) \in {\rm lPoly}_{A/I}^\wedge$, the right-hand side of \eqref{eq:botheqs} then participates in a further equivalence \begin{equation}\label{eq:anothereq}\Prism_{R\langle P \rangle^{\otimes^{\rm sat} \bullet + 1} / A} \xrightarrow{\simeq} \Prism_{R\langle P \rangle^{\otimes^{\rm sat} \bullet + 1} / A\langle P_{\rm perf}^{\oplus^{\rm ex} \bullet + 1} \rangle},\end{equation} as $A/I \langle P_{\rm perf}^{\oplus^{\rm ex} \bullet + 1} \rangle$ is a different choice of levelwise perfectoid mapping (indeed, levelwise surjecting) to the levelwise quasiregular semiperfectoid ring $R\langle P \rangle^{\otimes^{\rm sat} \bullet + 1}$. We make the following definition: 

\begin{definition}\label{def:nygaardfil} Let $(A, I)$ be a perfect prism and let $(R, P) \in {\rm lPoly}_{A/I}^\wedge$. With notation as in the proof of Theorem \ref{thm:prisminftyroot}, we define the \emph{Nygaard filtration} to be \[{\rm Fil}_N^{\ge i} \Prism_{(R, P) / A} := {\rm Tot}({\rm Fil}^{\ge i}_N\Prism_{R\langle P \rangle^{\otimes^{\rm sat} \bullet + 1} / A\langle P_{\rm perf}^{\oplus^{\rm ex} \bullet + 1} \rangle}),\] that is, the totalization of the cosimplicial object ${\rm Fil}^{\ge i}_N\Prism_{R\langle P \rangle^{\otimes^{\rm sat} \bullet + 1} / A\langle P_{\rm perf}^{\oplus^{\rm ex} \bullet + 1} \rangle}$. For a general $p$-complete animated pre-log $A/I$-algebra, we define the Nygaard filtration by left Kan extension. 
\end{definition}

We now work towards the proof of Theorem \ref{thm:nygaard}, largely inspired by the proof strategy of \cite[Theorem 13.1]{BS22}. Let us first recall from \cite[Construction 3.9]{BLMP24} that, given a perfectoid ring $A/I$, we define the functor \[{\rm Ani}({\rm lPoly}_{A/I}^\wedge) \to {\cal D}(A)^\wedge_{(p, I)}, \quad (R, P) \mapsto \widehat{\Prism}^{\rm nc}_{(R, P) / A}\]  by left Kan extension of Nygaard-completed log prismatic cohomology $\widehat{\Prism}_{(R, P) / A}$ from the log polynomial case. 

\begin{lemma}\label{lem:polynygaard} Let $(R, P) \in {\rm lPoly}_{A/I}^\wedge$ be a $p$-complete polynomial pre-log $A/I$-algebra. Then there is an equivalence ${\Prism}_{(R, P) / A} \xrightarrow{\simeq} \widehat{\Prism}_{(R, P) / A}^{\rm nc} $ of $(p, I)$-complete objects in ${\cal D}(A)$. 
\end{lemma}

\begin{proof} By Lemma \ref{lem:prisminftyroot} (and the equivalences \eqref{eq:botheqs} in particular), we may compute $\Prism_{(R, P) / A}$ as the totalization of the cosimplicial diagram $\Prism_{R\langle P \rangle^{\otimes^{\rm sat} \bullet + 1}}$. By \cite[Corollary 4.22]{BLMP24}, the same is true for $\widehat{\Prism}_{(R, P) / A}^{\rm nc}$.
\end{proof}

While left somewhat implicit in the proof, we find it worthwhile to highlight that the equivalence of Lemma \ref{lem:polynygaard} comes to life from the levelwise equivalences \[\Prism_{R\langle P \rangle^{\otimes^{\rm sat} \bullet + 1}} \xrightarrow{\simeq} \widehat{\Prism}^{\rm nc}_{R\langle P \rangle^{\otimes^{\rm sat} \bullet + 1}}\] of \cite[Theorem 13.1]{BS22} and saturated descent for both constructions. 

\begin{corollary}\label{cor:pcompleteani} For any $p$-complete animated pre-log $A/I$-algebra $(R, P)$, there is an equivalence $\Prism_{(R, P) / A} \xrightarrow{\simeq} \widehat{\Prism}^{\rm nc}_{(R, P) / A} $. \qed
\end{corollary}

We now match the various Nygaard filtrations in this picture. The Nygaard filtration ${\rm Fil}^{\ge i}_N \widehat{\Prism}^{\rm nc}_{(R, P) / A}$ is also defined by left Kan extension, so that its Nygaard completion recovers $\widehat{\Prism}_{(R, P) / A}$ for log quasiregular semiperfectoid $A/I$-algebras $(R, P)$. We observe that this coincides with the \emph{a priori} different approach obtained by imitating Definition \ref{def:nygaardfil}, obtained by setting \begin{equation}\label{eq:anothernygaard}\widetilde{{\rm Fil}}_N^{\ge i} \widehat{\Prism}_{(R, P) / A}^{\rm nc} := {\rm Tot}({\rm Fil}^{\ge i}_N \widehat{\Prism}^{\rm nc}_{R\langle P \rangle^{\otimes^{\rm sat} \bullet + 1} / A\langle P_{\rm perf}^{\oplus^{\rm ex} \bullet + 1} \rangle})\end{equation} for $(R, P) \in {\rm lPoly}_{A / I}^\wedge$ and left Kan extending in general. These filtrations coincide: indeed, the saturated descent of \cite[Theorem 4.18]{BLMP24}, and hence that of \cite[Corollary 4.22]{BLMP24}, are compatible with the Nygaard filtrations. 

\begin{lemma}\label{lem:intextnygaard} Let $(S, Q)$ be log quasiregular semiperfectoid. There is a map \[\Prism_{(S, Q)} \to \widehat{\Prism}_{(S, Q)}\] which exhibits the target as the Nygaard-completion of the source. 
\end{lemma}

\begin{proof} Let $(A, I)$ be a perfect prism with a map $A/I \to S$. By Corollary \ref{cor:pcompleteani}, there is an equivalence $\Prism_{(S, Q)} \xrightarrow{\simeq} \widehat{\Prism}_{(S, Q) / A}^{\rm nc}$. By construction, the filtration of Definition \ref{def:nygaardfil} is compatible with \eqref{eq:anothernygaard}. As noted above, the latter filtration coincides with the natural filtration ${\rm Fil}^{\ge i}_N \widehat{\Prism}^{\rm nc}_{(R, P) / A}$ for which the canonical map $\widehat{\Prism}^{\rm nc}_{(R, P) / A} \to \widehat{\Prism}_{(R, P) / A}$ is a completion by construction. This concludes the proof. 
\end{proof}

We expect there to be a variant of Lemma \ref{lem:intextnygaard} as functorial as \cite[Theorem 13.1]{BS22}. This would require a version of \cite[Lemma 13.2]{BS22} in the log setting, cf.\ \cite[Remark 4.15]{KY23}. 

\begin{proof}[Proof of Theorem \ref{thm:nygaard}] This follows from Lemma \ref{lem:intextnygaard} by unfolding \cite[Theorem 4.22]{BLPO23Prism}.
\end{proof}

\subsection{Comparison of Nygaard filtrations}\label{sec:nygaardcomp} We proceed to verify that the Nygaard filtration of Definition \ref{def:nygaardfil} indeed recovers that pursued by Koshikawa--Yao \cite[Definition 5.13, Section 5.5]{KY23}. Let us first recall the construction of the Nygaard filtration of \emph{loc.\ cit.}\ in a language more compatible with the present exposition:

\begin{construction} Let $(A, I)$ be a perfect prism and let $(R, P) \in {\rm lPoly}^\wedge_{A/I}$. There is a chain of equivalences \begin{equation}\label{eq:chainofeqs}\begin{tikzcd}[column sep = tiny, row sep = tiny]\Prism_{(R, P) / A}  \ar{d}{\simeq}  \\ \lim_\Delta \Prism_{(R \widehat{\otimes}_{A/I\langle P \rangle} A/I\langle P_{\rm perf}^{\oplus_P \bullet + 1} \rangle, P_{\rm perf}^{\oplus_P \bullet + 1}) / A} \ar{d}{\simeq}  \\ \lim_\Delta \Prism_{(R \widehat{\otimes}_{A/I\langle P \rangle} A/I\langle P_{\rm perf}^{\oplus_P \bullet + 1} \rangle, P_{\rm perf}^{\oplus_P \bullet + 1}) / (A\langle P_{\rm perf}^{\oplus \bullet + 1} \rangle, P_{\rm perf}^{\oplus \bullet + 1})} \ar{d}{\simeq} \\ \lim_\Delta  \Prism_{(R \widehat{\otimes}_{A/I\langle P \rangle} A/I\langle (P_{\rm perf}^{\oplus_P \bullet + 1} \rangle, P_{\rm perf}^{\oplus_P \bullet + 1}) / (A\langle (P_{\rm perf}^{\oplus \bullet + 1})^{\rm ex} \rangle, (P_{\rm perf}^{\oplus \bullet + 1})^{\rm ex})} \\ \lim_\Delta \Prism_{R \widehat{\otimes}_{A/I\langle P \rangle} A/I\langle (P_{\rm perf}^{\oplus_P \bullet + 1} \rangle / A\langle (P_{\rm perf}^{\oplus \bullet + 1})^{\rm ex} \rangle} \ar[swap]{u}{\simeq}\end{tikzcd}\end{equation} of Frobenius-equivariant $(p, I)$-complete objects in ${\cal D}(A)$. Indeed, the derived Hodge--Tate comparison \cite[Proposition 4.5]{KY23} reduces this to a question of log cotangent complexes, and in this case the equivalences can be established in a manner analogous to e.g. \cite[Proposition 4.20]{BLPO23Prism}. More explicitly, the equivalences come to life as follows:
\begin{enumerate}
\item The first equivalence from the top is quasisyntomic descent;
\item the second equivalence from the top is independence of (levelwise) perfectoid mapping to the levelwise quasiregular semiperfectoid \[(R \widehat{\otimes}_{A/I\langle P \rangle} A/I\langle P_{\rm perf}^{\oplus_P \bullet + 1} \rangle, P_{\rm perf}^{\oplus_P \bullet + 1})\] as well as the insensitivity of the log cotangent complex to perfect monoids (e.g. \cite[Corollary 4.18]{BLPO23Prism}); 
\item the bottom chain of equivalences is \cite[Proposition 4.20]{BLPO23Prism}, where $(P_{\rm perf}^{\oplus \bullet + 1})^{\rm ex}$ denotes the exactification of the canonical map $P_{\rm perf}^{\oplus \bullet + 1} \to P_{\rm perf}^{\oplus_P \bullet + 1}$. 
\end{enumerate}

\noindent Motivated by a similar analysis in their setting, Koshikawa--Yao define the (derived) Nygaard filtration \begin{equation}\label{eq:kynygaard}{\rm Tot}({\rm Fil}^{\ge i}_N\Prism_{R \widehat{\otimes}_{A/I\langle P \rangle} A/I\langle (P_{\rm perf}^{\oplus_P \bullet + 1} \rangle / A\langle (P_{\rm perf}^{\oplus \bullet + 1})^{\rm ex} \rangle}^{(1)}).\end{equation} We observe that the Frobenius-twist is essential, as the ``exactified ring'' $A\langle (P_{\rm perf}^{\oplus \bullet + 1})^{\rm ex} \rangle$ will in general not be perfectoid (see e.g.\ \cite[Examples 4.21, 4.22]{BLPO23Prism}). We also note that there is an explicit comparison map \begin{equation}\label{eq:comparison}\Prism_{R \widehat{\otimes}_{A/I\langle P \rangle} A/I\langle (P_{\rm perf}^{\oplus_P \bullet + 1} \rangle / A\langle (P_{\rm perf}^{\oplus \bullet + 1})^{\rm ex} \rangle} \xrightarrow{} \Prism_{R\langle P \rangle^{\otimes^{\rm sat} \bullet + 1} / A\langle P_{\rm perf}^{\oplus^{\rm ex} \bullet + 1} \rangle}\end{equation} which is an equivalence by the proof of Theorem \ref{thm:prisminftyroot} and the equivalences \eqref{eq:chainofeqs}. 
\end{construction}

Koshikawa--Yao \cite[Lemma 5.12 and the ensuing discussion]{KY23} show that the Frobenius-twisted log prismatic cohomology $\Prism_{R \widehat{\otimes}_{A/I\langle P \rangle} A/I\langle (P_{\rm perf}^{\oplus_P \bullet + 1} \rangle / A\langle (P_{\rm perf}^{\oplus \bullet + 1})^{\rm ex} \rangle}^{(1)}$ may be computed as $\Prism_{R \widehat{\otimes}_{A/I\langle P \rangle} A/I\langle (P_{\rm perf}^{\oplus_P \bullet + 1} \rangle / A\langle (P_{\rm perf}^{\oplus \bullet + 1})^{\rm ex} \rangle} \otimes_{A, \phi} A$. Hence the Nygaard filtration of \eqref{eq:kynygaard} may be considered one of $\Prism^{(1)}_{(R, P) / A}$ by \eqref{eq:chainofeqs}. 

\begin{proposition}\label{prop:nygaardcomp} After base-change along the Frobenius on $A$, the map \eqref{eq:comparison} is an equivalence \[\Prism_{R \widehat{\otimes}_{A/I\langle P \rangle} A/I\langle (P_{\rm perf}^{\oplus_P \bullet + 1} \rangle / A\langle (P_{\rm perf}^{\oplus \bullet + 1})^{\rm ex} \rangle} \otimes_{A, \phi} A \xrightarrow{} \Prism_{R\langle P \rangle^{\otimes^{\rm sat} \bullet + 1} / A\langle P_{\rm perf}^{\oplus^{\rm ex} \bullet + 1} \rangle} \otimes_{A, \phi} A\] of filtered objects, where the source carries the filtration \eqref{eq:kynygaard}, while the target carries that of Definition \ref{def:nygaardfil}. 
\end{proposition}

\begin{proof}
By the very construction of the Nygaard filtration of \cite{KY23}, the equivalences \eqref{eq:chainofeqs} are filtered. However, the saturated descent equivalence is also filtered with respect to Definition \ref{def:nygaardfil} by construction, so the map is indeed one of filtered objects. It is an equivalence of such, as for $(R, P) \in {\rm lPoly}_{A/I}^\wedge$, the graded pieces in both cases are given by truncated Frobenius-twisted log Hodge--Tate cohomology $\tau_{\le i} \overline{\Prism}_{(R, P) / A}\{i\}$ (and are thus complete). 
\end{proof}
\section{Log syntomic cohomology as a left Kan extension}\label{sec:logsynleftkan}

In this section, we apply saturated descent for log syntomic cohomology to deduce an analog of \cite[Theorem 5.1(2)]{AMMN22}. Recall that ${\rm lQSyn}^{\rm qsyn}$ denotes the full subcategory of ${\rm lQSyn}$ spanned by those $(S, Q)$ with $S \widehat{\otimes}_{{\Bbb Z}\langle Q \rangle} {\Bbb Z}\langle Q_{\rm perf} \rangle$ quasisyntomic with $Q$ fine, saturated, and $p$-torsionfree. In particular, ${\rm lQSyn}^{\rm qsyn}$ contains ${\rm lPoly}^\wedge_{{\Bbb Z}_p}$.
Recall from \cite[Theorem 1.3(4)]{BLPO23Prism} that ${\Bbb Z}_p(i)(S, Q)$ can be defined in terms of the motivic filtrations on log ${\rm TC}^-$ and ${\rm TP}$.

\begin{theorem}\label{thm:logsynleftkan} The functors $(S, Q) \mapsto {\Bbb Z}_p(i)(S, Q)$ and $(S, Q) \mapsto {\rm Fil}_N^{\ge i}{\rm TC}((S, Q) ; {\Bbb Z}_p)$ on ${\rm lQSyn}^{\rm qsyn}$ are left Kan extended from $p$-complete polynomial pre-log rings.
\end{theorem}

Recall that, using the analogous statement in the non-log setting, one extends ${\Bbb Z}_p(i)(-)$ to a functor \begin{equation}\label{eq:syntomicani} {\Bbb Z}_p(i)(-) \colon {\rm Ani}({\rm CRing}^\wedge_{{\Bbb Z}_p}) \to D({\Bbb Z}_p)\end{equation} on $p$-complete animated ${\Bbb Z}_p$-algebras. In order for this construction to play well with our application of saturated descent, we need the following elementary observation:

\begin{lemma}\label{lem:levelwiseperf} Let $Q$ be a commutative monoid and let $Q_{\bullet} \xrightarrow{\simeq} Q$ be a simplicial resolution of $Q$ by free commutative monoids. The induced map \[Q_{\bullet, {\rm perf}} \xrightarrow{\simeq} Q_{\rm perf}\] is still an equivalence, where $Q_{\bullet, {\rm perf}}$ is the levelwise direct limit perfection of $Q_{\bullet}$. 
\end{lemma}

\begin{proof} The direct limit perfection is a filtered colimit. 
\end{proof}

\begin{proof}[Proof of Theorem \ref{thm:logsynleftkan}] Let us first apply saturated descent for ${\Bbb Z}_p(i)(-, -)$ to obtain an equivalence \[{\Bbb Z}_p(i)(S, Q) \xrightarrow{\simeq} {\rm Tot}(\!\!\!\begin{tikzcd}[column sep = tiny] {\Bbb Z}_p(i)(S \widehat{\otimes}_{{\Bbb Z}_p \langle Q \rangle} {\Bbb Z}_p\langle Q_{\rm perf} \rangle) \ar[shift left = 0.5]{r} \ar[shift right = 0.5]{r} \!\! &\!  {\Bbb Z}_p(i)(S \widehat{\otimes}_{{\Bbb Z}_p \langle Q \rangle} {\Bbb Z}_p \langle Q_{\rm perf} \oplus_Q^{\rm sat} Q_{\rm perf}\rangle) \cdots \end{tikzcd}\!\!\!).\] Suppose now that $(S_{\bullet}, Q_{\bullet}) \xrightarrow{\simeq} (S, Q)$ is a simplicial resolution of $(S, Q)$ that is levelwise a $p$-complete polynomial pre-log ${\Bbb Z}_p$-algebra. We then observe that the canonical map $S_{\bullet} \widehat{\otimes}_{{\Bbb Z}_p \langle Q_{\bullet} \rangle} {\Bbb Z}_p \langle Q_{\bullet, {\rm perf}} \rangle \xrightarrow{} S \widehat{\otimes}_{{\Bbb Z}_p \langle Q \rangle} {\Bbb Z}_p \langle Q_{\rm perf} \rangle$ is an equivalence, by Lemma \ref{lem:levelwiseperf} and using $p$-complete flatness of ${\Bbb Z}_p \langle Q \rangle \to S$. More generally, we claim that the canonical map \begin{equation}\label{eq:notpoly} S_{\bullet} \widehat{\otimes}_{{\Bbb Z}_p\langle Q_{\bullet} \rangle} {\Bbb Z}_p\langle Q_{\bullet, {\rm perf}}^{\oplus_{Q_{\bullet}}^{{\rm sat}, {\bullet + 1}}} \rangle \xrightarrow{}  S \widehat{\otimes}_{{\Bbb Z}_p\langle Q \rangle} {\Bbb Z}_p\langle Q_{{\rm perf}}^{\oplus_{Q}^{{\rm sat}, {\bullet + 1}}} \rangle \end{equation} is an equivalence, where the saturation on the source is formed levelwise. In order to see this, we combine the already established equivalence in cosimplicial degree zero with the isomorphisms $Q_{\rm perf}^{\oplus^{\rm sat, {\bullet + 1}}} \cong Q_{\rm perf} \oplus (Q_{\rm perf}^{\rm gp}/Q^{\rm gp})^{\oplus \bullet}$ and the analogous one on the resolutions. This rewrites the map \eqref{eq:notpoly} as \begin{equation}\label{eq:notpoly2} \begin{tikzcd}[row sep = small] (S_{\bullet} \widehat{\otimes}_{{\Bbb Z}_p\langle Q_{\bullet} \rangle} {{\Bbb Z}_p\langle  Q_{\bullet, {\rm perf}}\rangle}) \widehat{\otimes}_{{\Bbb Z}_p} {\Bbb Z}_p \langle  Q^{\rm gp}_{\bullet, {\rm perf}} / Q^{\rm gp}_{\bullet} \rangle \ar{d}{\simeq} \\ (S \widehat{\otimes}_{{\Bbb Z}_p\langle Q \rangle} {{\Bbb Z}_p\langle  Q_{{\rm perf}}\rangle}) \widehat{\otimes}_{{\Bbb Z}_p} {\Bbb Z}_p \langle  Q^{\rm gp}_{{\rm perf}} / Q^{\rm gp} \rangle \end{tikzcd}\end{equation} From this, the equivalence \eqref{eq:notpoly} follows from the equivalence $Q_{\bullet, {\rm perf}}^{\rm gp}/Q_{\bullet}^{\rm gp} \to Q^{\rm gp}_{\rm perf}/Q^{\rm gp}$ obtained from Lemma \ref{lem:levelwiseperf}. 

Applying the functor \eqref{eq:syntomicani} to the equivalence \eqref{eq:notpoly}, we thus obtain a chain of equivalences \[\begin{tikzcd}[row sep = tiny]{\Bbb Z}_p(i)(S, Q) \ar{d}{\simeq} \\ \lim_{\Delta}( {\Bbb Z}_p(i)(S \widehat{\otimes}_{{\Bbb Z}_p\langle Q \rangle} {\Bbb Z}_p\langle Q_{{\rm perf}}^{\oplus_{Q}^{{\rm sat}, {\bullet + 1}}} \rangle)) \\ \lim_{\Delta}({\Bbb Z}_p(i)(S_{\bullet} \widehat{\otimes}_{{\Bbb Z}_p\langle Q_{\bullet} \rangle} {\Bbb Z}_p\langle Q_{\bullet, {\rm perf}}^{\oplus_{Q_{\bullet}}^{{\rm sat}, {\bullet + 1}}} \rangle ))\ar[swap]{u}{\simeq}\end{tikzcd}\] and the target is levelwise equivalent to ${\Bbb Z}_p(i)(S_{\bullet}, Q_{\bullet})$ by saturated descent. Here, ${\Bbb Z}_p(i)(S_{\bullet}, Q_{\bullet})$ is merely the levelwise application of ${\Bbb Z}_p(i)(-, -)$ on the simplicial pre-log ring $(S_{\bullet}, Q_{\bullet})$.  

So far, we have proven that ${\Bbb Z}_p(i)(S, Q)$ can be computed in terms of any simplicial resolution $(S_{\bullet}, Q_{\bullet})$ of $p$-complete polynomial pre-log ${\Bbb Z}_p$-algebras. To argue that the construction $(S_{\bullet}, Q_{\bullet}) \mapsto {\Bbb Z}_p(i)(S_{\bullet}, Q_{\bullet})$ commutes with sifted colimits of such, we once again appeal to saturated descent: The functor \eqref{eq:syntomicani} commutes with sifted colimits, and so does the formation of each term in the saturated \v{C}ech nerve. Indeed, the construction $Q \mapsto Q_{\rm perf}^{\oplus_Q^{\rm sat} n + 1} \cong Q_{\rm perf} \oplus (Q_{\rm perf}^{\rm gp} / Q^{\rm gp})^{\oplus n}$ commutes with colimits, and so does the base-change. We use \cite[Proof of Lemma 4.9]{BLMP24} to commute the cosimplicial limit past the sifted colimit: \emph{Loc.\ cit.}\ gives the analogous assertion for the cotangent complex, from which we deduce it for prismatic cohomology and its Nygaard filtration, and hence for syntomic cohomology. This concludes the proof of the statement about ${\Bbb Z}_p(i)(-, -)$, while that for ${\rm Fil}_N^{\ge i}{\rm TC}((-, -) ; {\Bbb Z}_p)$ follows from Theorem \ref{thm:logtcleftkan}. 
\end{proof}

\section{A log variant of the Beilinson fiber square}\label{sec:logbeilinson} In this final section, we prove Theorems \ref{thm:logbeilinson} and \ref{thm:diaoyao} from the  introduction.  Recall from \cite[Section 2]{AMMN22} (and in particular \cite[Proof of Theorem 2.12]{AMMN22}) that the ${\rm TC}$-variant of the Beilinson fiber square comes to life by studying the effect of the map $X \otimes_{\Bbb S} {\Bbb Z}^{\rm triv} \to X \otimes_{{\Bbb S}} {\rm THH}({\Bbb F}_p)$ on ${\rm TC}$ and ${\rm TC}^-$ (cf.\ \cite[Proposition 2.8]{AMMN22}). Unsurprisingly, the proof of Theorem \ref{thm:logbeilinson} arises from their results applied to the cyclotomic spectrum $X = {\rm THH}(R, P)$. To this end, we begin with the following analog of \cite[Theorem 2.12]{AMMN22}:

\begin{proposition}\label{prop:prelimsquare} Let $(R, P)$ be a pre-log ring. There is a square of the form  \[\begin{tikzcd}{\rm TC}((R, P) ; {\Bbb Z}_p) \ar{r} \ar{d} & {\rm TC}((R \otimes_{\Bbb S} {\Bbb F}_p, P) ; {\Bbb Z}_p) \ar{d} \\ {\rm HC}^-((R, P) ; {\Bbb Z}_p) \ar{r} & {\rm HP}((R, P) ; {\Bbb Z}_p).\end{tikzcd}\] which becomes cartesian after inverting $p$. 
\end{proposition}

\begin{proof} This follows from replacing ${\rm THH}(R ; {\Bbb Z}_p)$ with ${\rm THH}((R, P) ; {\Bbb Z}_p)$ in \cite[Proof of Theorem 2.12]{AMMN22} and using basic properties of log ${\rm THH}$ established in \cite{BLPO23Prism}. We spell this out below. 

We build a commutative square \[\begin{tikzcd}{\rm TC}((R, P) ; {\Bbb Z}_p) \ar{d} \\ {\rm TC}({\rm THH}(R, P) \otimes_{\Bbb S} {\Bbb Z}^{\rm triv} ; {\Bbb Z}_p) \ar{r} \ar{d} & {\rm TC}((R \otimes_{{\Bbb S}} {\Bbb F}_p, P) ; {\Bbb Z}_p) \ar{d} \\ ({\rm THH}((R, P) ; {\Bbb Z}_p) \otimes_{{\Bbb S}} {\Bbb Z}^{\rm triv})^{hS^1} \ar{r} \ar{d} & ({\rm THH}((R, P) ; {\Bbb Z}_p) \otimes_{{\Bbb S}} {\Bbb Z}^{\rm triv})^{tS^1} \ar{d} \\ {\rm HC}^-((R, P) ; {\Bbb Z}_p) \ar{r} & {\rm HP}((R, P) ; {\Bbb Z}_p).\end{tikzcd}\] We claim that both squares are cartesian after inverting $p$, and that the top vertical arrow is an equivalence after inverting $p$, yielding the result. 

 We first treat the top square. There is an equivalence \[{\rm TC}((R \otimes_{\Bbb S} {\Bbb F}_p, P) ; {\Bbb Z}_p) \simeq {\rm TC}({\rm THH}(R, P) \otimes_{{\Bbb S}} {\rm THH}({\Bbb F}_p) ; {\Bbb Z}_p)\] by Lemma \ref{lem:commuteswithcolims} (and the identification ${\rm THH}({\Bbb S}) \simeq {\Bbb S}^{\rm triv}$), so the upper square is that of \cite[Corollary 2.10]{AMMN22} (with $X = {\rm THH}(R, P)$). Hence \cite[Corollary 2.10]{AMMN22} applies to say that the upper square is cartesian after inverting $p$. 

We now treat the bottom square. It is induced by the $S^1$-equivariant map ${\rm TC}((R, P) ; {\Bbb Z}_p) \otimes_{\Bbb S} {\Bbb Z}^{\rm triv} \to {\rm HH}((R, P) ; {\Bbb Z}_p)$. For the latter map, we recall that ${\rm THH}(R, P) \otimes_{{\rm THH}({\Bbb Z})} {\Bbb Z} \simeq {\rm HH}(R, P)$ by \cite[Corollary 3.5]{BLPO23Prism}. To see that it becomes cartesian after inverting $p$, we study the induced map of horizontal fibers. By \cite[Corollary I.4.3]{NS18}, this is the map \[({\rm THH}((R, P) ; {\Bbb Z}_p) \otimes_{{\Bbb S}} {\Bbb Z}^{\rm triv})_{hS^1}[1] \to {\rm HH}((R, P) ; {\Bbb Z}_p)_{hS^1}[1].\] But this is an equivalence after inverting $p$ since ${\rm THH}((R, P) ; {\Bbb Z}_p) \otimes_{{\Bbb S}} {\Bbb Z} \to {\rm HH}((R, P) ; {\Bbb Z}_p)$ is. Hence, the lower square is cartesian after inverting $p$.

The top vertical arrow is induced by the map ${\rm THH}((R, P) ; {\Bbb Z}_p) \to {\rm THH}((R, P) ; {\Bbb Z}_p) \otimes_{{\Bbb S}} {\Bbb Z}^{\rm triv}$ of cyclotomic spectra. By \cite[Remark 2.4]{AMMN22}, this map is equivalent to the canonical map ${\rm TC}((R, P) ; {\Bbb Z}_p) \to {\rm TC}((R, P) ; {\Bbb Z}_p) \otimes_{{\Bbb S}} {\Bbb Z}$, which is an equivalence after inverting $p$. This concludes the proof. 
\end{proof}

\begin{proof}[Proof of Theorem \ref{thm:logbeilinson}] By Proposition \ref{prop:prelimsquare}, it suffices to prove that the map \begin{equation}\label{eq:isaquasiisog}{\rm TC}((R \otimes_{{\Bbb S}} {\Bbb F}_p, P) ; {\Bbb Q}_p) \xrightarrow{} {\rm TC}((R/p, P) ; {\Bbb Q}_p)\end{equation} is an equivalence. This will follow once we prove that \[{\rm TC}((R \otimes_{\Bbb S} {\Bbb F}_p, P) ; {\Bbb Z}_p) \xrightarrow{} {\rm TC}((R/p, P) ; {\Bbb Z}_p)\] is a quasi-isogeny of spectra. For this, it suffices to prove that \[{\rm THH}((R \otimes_{\Bbb S} {\Bbb F}_p, P) ; {\Bbb Z}_p) \to {\rm THH}((R/p, P) ; {\Bbb Z}_p)\] is a quasi-isogeny of cyclotomic spectra. This map is the base-change of the quasi-isogeny ${\rm THH}(R \otimes_{\Bbb S} {\Bbb F}_p ; {\Bbb Z}_p) \to {\rm THH}(R/p ; {\Bbb Z}_p)$ along the repletion map ${\Bbb S}_p[B^{\rm cyc}(P)] \to {\Bbb S}_p[B^{\rm rep}(P)]$. Keeping in mind that connectivity in the cyclotomic $t$-structure is checked on underlying spectra, $- \widehat{\otimes}_{{\Bbb S}_p[B^{\rm cyc}(P)]} {\Bbb S}_p[B^{\rm rep}(P)]$ is right $t$-exact, and so \eqref{eq:isaquasiisog} is a quasi-isogeny (cf.\ \cite[Variant 2.4]{AN21} and \cite[Proof of Lemma 3.11]{AMMN22}). 
\end{proof}

\begin{proof}[Proof of Theorem \ref{thm:diaoyao}] By unfolding, it suffices to construct the square for log quasiregular semiperfectoid pre-log rings $(S, Q)$. In this situation, the pullback square \[\begin{tikzcd}{\rm TC}((S, Q) ; {\Bbb Q}_p) \ar{r} \ar{d}& {\rm TC}((S/p, Q) ; {\Bbb Q}_p) \ar{d} \\ {\rm HC}^-((S, Q), {\Bbb Q}_p) \ar{r} & {\rm HP}((S, Q) ; {\Bbb Q}_p)\end{tikzcd}\] of Theorem \ref{thm:logbeilinson} has terms concentrated in even degrees by \cite[Theorems 1.3, 1.8]{BLPO23Prism}. Applying $\tau_{[2i - 1, 2i]}$, then, we obtain the pullback square \begin{equation}\label{hodgepullback}\begin{tikzcd}{\Bbb Q}_p(i)(S, Q) \ar{r} \ar{d} & {\Bbb Q}_p(i)(S/p, Q) \ar{d} \\ ({\widehat{L\Omega}_{(S, Q)/{\Bbb Z}_p}^{\ge i}}) \otimes_{{\Bbb Z}_p} {\Bbb Q}_p \ar{r} & ({\widehat{L\Omega}_{(S, Q)/{\Bbb Z}_p}}) \otimes_{{\Bbb Z}_p} {\Bbb Q}_p,\end{tikzcd}\end{equation} compare \cite[Theorem 13.8]{DY24}. It remains to show that the same is true without Hodge-completion. To do so, we recall that the equivalence ${\rm TC}((S \otimes_{\Bbb S} {\Bbb F}_p, Q), {\Bbb Q}_p) \to {\rm TC}((S/p, Q) ; {\Bbb Z}_p)$ arises from a slightly more structured statement, namely that the map ${\rm TC}((S \otimes_{\Bbb S} {\Bbb F}_p, Q) ; {\Bbb Z}_p) \to {\rm TC}((S/p, Q) ; {\Bbb Z}_p)$ is a quasi-isogeny.  Thus, upon applying $\pi_{2i}$ and \cite[Theorems 1.3, 1.8]{BLPO23Prism}, we obtain a diagram \[\begin{tikzcd}{\Bbb Z}_p(i)(S, Q) \ar[shift left = 1, dashed]{r} & \pi_{2i}{\rm TC}((S \otimes_{\Bbb S} {\Bbb F}_p, Q) ; {\Bbb Z}_p) \ar{l} \ar{r} & \widehat{L\Omega}_{(S, Q) / {\Bbb Z}_p}.\end{tikzcd}\] Here the dashed morphism is a witness that the left-hand solid morphism is a quasi-isogeny, that is, both composites with the left-hand solid morphism are integer multiples of the identity. As ${\Bbb Z}_p(i)(-, -)$ is left Kan extended from finitely generated $p$-complete polynomial pre-log rings by Theorem \ref{thm:logsynleftkan}, we see that the morphism ${\Bbb Z}_p(i)(S, Q) \to \widehat{L\Omega}_{(S, Q) / {\Bbb Z}_p}^{\ge i}$ factors through a morphism ${\Bbb Z}_p(i)(S, Q) \to L\Omega_{(S, Q) / {\Bbb Z}_p}^{\ge i}$, and upon inverting $p$ we obtain a maps \[{\Bbb Q}_p(i)(S, Q) \to (L\Omega^{\ge i}_{(S, Q) / {\Bbb Z}_p}) \otimes_{{\Bbb Z}_p} {\Bbb Q}_p \text{  and  }{\Bbb Q}_p(i)(S/p, Q) \to ({L\Omega_{(S, Q)/{\Bbb Z}_p}}) \otimes_{{\Bbb Z}_p} {\Bbb Q}_p\] which is isomorphic to the vertical maps of \eqref{hodgepullback} once post-composed with the Hodge-completions.

\end{proof}

\end{document}